\documentclass[11pt]{article}

\usepackage[T1]{fontenc}
\usepackage[utf8]{inputenc}
\usepackage{lmodern}
\usepackage{amsmath,amssymb,amsthm,mathrsfs}
\usepackage[margin=1in]{geometry}
\usepackage{microtype}
\usepackage{parskip}
\usepackage{needspace}
\usepackage{xcolor}
\usepackage{mathtools,enumitem}
\usepackage[hidelinks]{hyperref}
\hypersetup{
 pdftitle={A planar algebraic Zarankiewicz theorem},
 pdfauthor={},
 pdfsubject={Incidences over arbitrary fields and Boolean combinations of polynomial equations}
}
\allowdisplaybreaks
\newtheorem{theorem}{Theorem}[section]
\newtheorem{lemma}[theorem]{Lemma}
\newtheorem{proposition}[theorem]{Proposition}
\newtheorem{corollary}[theorem]{Corollary}
\theoremstyle{definition}

\theoremstyle{remark}
\newtheorem{remark}[theorem]{Remark}
\newcommand{\F}{\mathbb{F}}
\newcommand{\kk}{\Bbbk}
\newcommand{\Fp}{\mathbb{F}_p}
\newcommand{\CC}{\mathcal C}
\newcommand{\family}{\mathfrak C}
\newcommand{\Der}{\mathscr D}
\newcommand{\ord}{\operatorname{ord}}
\newcommand{\Res}{\operatorname{Res}}
\newcommand{\Disc}{\operatorname{Disc}}
\newcommand{\Iess}{I^{\!*}}
\newcommand{\Zess}{Z^{\!*}}
\newcommand{\epspart}[1]{\left(#1\right)_+}
\AddToHook{env/theorem/before}{\Needspace{5\baselineskip}}
\AddToHook{env/lemma/before}{\Needspace{5\baselineskip}}
\AddToHook{env/proposition/before}{\Needspace{5\baselineskip}}
\AddToHook{env/corollary/before}{\Needspace{5\baselineskip}}

\title{A planar algebraic Zarankiewicz theorem over prime fields}
\author{Le Quang-Ham\thanks{The Viet Nam National Institute of Educational Sciences. ~Email: {\tt hamlq@vnies.edu.vn}}}
\date{}

\begin{document}
\maketitle
\begin{abstract}
We prove an incidence bound for bipartite graphs on finite subsets of
$\F^2\times \F^2$ defined by Boolean combinations of polynomial equations
of bounded degree. If such a graph is $K_{k,k}$-free and its vertex classes
have sizes $m$ and $n$, then it has
$O_{t,k}((mn)^{2/3}+m+n+mn/p)$ edges, where $t$ bounds the description
complexity, $p$ is the characteristic of $\F$, and $1/p=0$ in characteristic
zero. We also prove this bound for incidences between points and distinct
geometrically irreducible components of a two-parameter polynomial
family, allowing singular and nonreduced members.
The proof extends Lewko's interpolation and contact-multiplicity method
from lines to algebraic families. Applications include rich components,
polynomial values on difference sets, and polynomial expansion.
\end{abstract}

\section{Introduction}

For finite sets of points $P$ and lines $\mathcal L$ in a plane, write
\[
 I(P,\mathcal L)=|\{(q,\ell)\in P\times\mathcal L:q\in\ell\}|.
\]
The celebrated Szemer\'edi--Trotter theorem \cite{ST} states that, for $m$ points and
$n$ lines in $\mathbb R^2$,
\[
 I(P,\mathcal L)\ll (mn)^{2/3}+m+n.
\]
This estimate is sharp up to an absolute constant. It has led to many
extensions involving algebraic curves and graphs defined by geometric
relations.

In a recent remarkable paper \cite{Lewko}, Lewko proved the corresponding estimate
over an arbitrary field $\F$:
\begin{equation}\label{eq:lewko}
 I(P,\mathcal L)\ll (mn)^{2/3}+m+n+\frac{mn}{p}.
\end{equation}
Here and throughout the paper, $p=\operatorname{char}\F$ in positive
characteristic, while $p=\infty$ and $1/p=0$ in characteristic zero.
Lewko's proof uses polynomial interpolation and bounds for orders of
contact. It applies without an order on the field or a partition into
cells. In this paper, we extend this approach to families of algebraic curves and to
graphs defined by Boolean combinations of polynomial equations.

Before Lewko's work, Bourgain, Katz and Tao \cite[Theorem~6.2]{BKT}
obtained a power saving over the elementary $O(N^{3/2})$ incidence bound
in prime fields. For $N=p^\alpha$, where $0<\alpha<2$ is fixed, their
bound for at most $N$ points and $N$ lines is
$O_\alpha(N^{3/2-\varepsilon(\alpha)})$ for some
$\varepsilon(\alpha)>0$.
Explicit improvements were subsequently obtained by Helfgott and Rudnev
\cite{HR} and Jones \cite{Jones}.
Over arbitrary fields, Stevens and de Zeeuw \cite[Theorem~3]{SdZ} proved
\[
 I(P,\mathcal L)\ll m^{11/15}n^{11/15}
\]
when $m^{7/8}<n<m^{8/7}$ and $m^{-2}n^{13}\ll p^{15}$.
In particular, this gives $O(N^{22/15})$ when
$m=n=N\ll p^{15/11}$.

For large sets over finite fields, a different estimate is well-known.
Vinh \cite[Theorem~3]{Vinh} proved that, over every finite
field $\mathbb F_{q_0}$,
\[
 I(P,\mathcal L)\le\frac{mn}{q_0}+\sqrt{q_0mn}.
\]
This holds for every prime power $q_0$, with no restriction on $m$ or $n$.
In particular, it gives $I(P,\mathcal L)=O(mn/q_0)$ when $mn\ge q_0^3$.
Here the denominator is the field order $q_0$, whereas the last term
in~\eqref{eq:lewko} depends on the characteristic.

For distinct geometrically irreducible plane curves of degree at most
$t$, B\'ezout's theorem and Cauchy--Schwarz give
$I(P,\CC)\ll_t m^{1/2}n+m$, which is $O_t(N^{3/2})$ when $m=n=N$.
Bounds for arbitrary sets of irreducible conics over prime fields were
obtained by Mohammadi, Pham and Warren \cite{MPW}. Their families need
not have two parameters, so their results apply to different
classes of curves and in different ranges. For large sets, Fourier
estimates for spheres
\cite{IR} and the point--sphere incidence theorem of Cilleruelo,
Iosevich, Lund, Roche-Newton and Rudnev \cite{CILRR} give
further bounds. Phuong, Pham and Vinh \cite{PPV} proved incidence
bounds for generalized spheres defined by diagonal polynomials.
Koh and Pham \cite{KP} obtained point--sphere bounds in odd dimensions
under conditions on the field and the radii. These estimates depend
on the geometry of the particular families.

\subsection{Semi-algebraic graphs and the main results}

A bipartite graph $G=(P,Q,E)$ with $P,Q\subset\mathbb R^2$ is
semi-algebraic of description complexity at most $t$ if membership in
$E$ is determined by a Boolean combination of at most $t$ polynomial
inequalities in the four coordinates, each of total degree at most $t$.
The graph is $K_{k,k}$-free if there are no $k$ vertices in each class
with all $k^2$ edges between them.
Fox, Pach, Sheffer, Suk and Zahl \cite[Theorem~1.1]{FPSSZ} proved that
every such $K_{k,k}$-free graph, with $|P|=m$ and $|Q|=n$, satisfies
\begin{equation}\label{eq:real}
 |E|\le C(t,k)\bigl((mn)^{2/3}+m+n\bigr).
\end{equation}
Thus, the planar semi-algebraic bound has the same exponent as the
point--line theorem.

Over an arbitrary field, we use Boolean combinations of polynomial
equations. Milojevi\'c, Sudakov and Tomon
\cite[Theorem~1.8]{MST} proved Zarankiewicz-type bounds for algebraic
graphs over arbitrary fields. For $K_{k,k}$-free graphs on two planar
vertex sets of size $N$, their result gives $O_{t,k}(N^{3/2})$ edges.
Our main result extends~\eqref{eq:real} to this setting, with the
additional characteristic term in Lewko's point--line bound. In
particular, we obtain $O_{t,k}(N^{4/3})$ when $N\le p^{3/2}$.

Fix positive integers $t$ and $k$. We say that a bipartite graph on finite
sets $P,Q\subset \F^2$ has \emph{algebraic description complexity at most
$t$} if there are $s\le t$ polynomials
$f_i\in \F[x_1,x_2,y_1,y_2]$ of total degree at most $t$ and a Boolean
function $\Psi:\{0,1\}^s\to\{0,1\}$ such that
\[
 (x,y)\in E\quad\Longleftrightarrow\quad
 \Psi\bigl([f_1(x;y)=0],\ldots,[f_s(x;y)=0]\bigr)=1.
\]
The allowed predicates are $f_i=0$ and $f_i\ne0$; no order or sign
condition is used.

Our first theorem reads as follows.
\begin{theorem}[Boolean algebraic bound]\label{thm:main}
There is a constant $C(t,k)$ such that every $K_{k,k}$-free graph of algebraic
description complexity at most $t$, with $|P|=m$ and $|Q|=n$, satisfies
\begin{equation}\label{eq:main}
 |E|\le C(t,k)\left((mn)^{2/3}+m+n+\frac{mn}{p}\right).
\end{equation}
One may take $C(t,k)\le C_{\mathrm{abs}}k^2t^{13}$ for an absolute constant $C_{\mathrm{abs}}$.
In particular, this holds for $P,Q\subset\Fp^2$.
\end{theorem}

For $f\in \F[x_1,x_2,y_1,y_2]$, let $\family_{\F}(f)$ consist
of the geometrically irreducible plane curves that occur as components of
\[
 Z_y(f)=\{x\in\overline{\F}^2:f(x;y)=0\},\qquad
 y\in \F^2,\quad f(\,\cdot\,;y)\not\equiv0.
\]
Thus, identically zero fibers are excluded, but nonzero fibers may be
reducible or nonreduced.
For an extension field $L/\F$, we define $\family_L(f)$ in the same
way, allowing parameters in $L^2$ and components over $\overline L$.

Our second theorem is the following.

\begin{theorem}[Distinct components of a two-parameter family]\label{thm:pure}
For every $t$, there is a constant $C_t$ with the following property. If
$\deg f\le t$, $P\subset \F^2$ has $m$ elements, and $\CC\subset\family_{\F}(f)$
consists of $n$ distinct curves, then
\begin{equation}\label{eq:pure}
 I(P,\CC):=\sum_{\Lambda\in\CC}|P\cap\Lambda|
 \le C_t\left((mn)^{2/3}+m+n+\frac{mn}{p}\right).
\end{equation}
One may take $C_t\le C_{\mathrm{abs}}t^{11}$ for an absolute constant $C_{\mathrm{abs}}$.
\end{theorem}

The point--line estimate~\eqref{eq:lewko} follows by taking
$f(x_1,x_2;a,b)=x_2-a x_1-b$, which parameterizes the nonvertical lines
and has total degree $2$. Distinct parameters give distinct lines, and
the incidence graph is $K_{2,2}$-free. Vertical lines contribute at most
$m$ further incidences. The characteristic term is already necessary
in this example: the $p^2$ points of $\Fp^2$ and the $p^2$ lines
$x_2=ax_1+b$ determine $p^3$ incidences.

Over $\mathbb R$, Theorem~\ref{thm:main} is a special case of the
semi-algebraic theorem of Fox, Pach, Sheffer, Suk and Zahl
\cite{FPSSZ}. For complex curves with two degrees of freedom,
Sheffer, Szab\'o and Zahl \cite[Theorem~1.3]{SSZ} proved the bound
$O_{\varepsilon,t,s}(m^{2/3+\varepsilon}n^{2/3}+m+n)$, where $s$
bounds both the number of curves through two points and the number
of common points of two curves. A two-parameter polynomial family
need not satisfy this degrees-of-freedom condition; see
Remark~\ref{rem:twoparameters}. Our structural lemmas allow us to treat
these families. The proof is
algebraic and applies over arbitrary fields. In positive characteristic,
it gives the exponent $2/3$ with the additional term $mn/p$.

Related bounds for zeros of non-Cartesian polynomials on products of
two-dimensional sets were proved by Nassajian Mojarrad, Pham,
Valculescu and de Zeeuw \cite{NMPVdZ}, with an $\varepsilon$ loss over
$\mathbb C$ and no such loss over $\mathbb R$. Both their result and
Theorem~\ref{thm:pure} include point--line incidences.
Section~\ref{sec:cartesian} explains the distinction between counting
parameters and counting distinct irreducible components, and deduces
the non-Cartesian bound without an $\varepsilon$ loss, with the
additional term $mn/p$.

Section~\ref{sec:applications} gives bounds for rich components,
translates of a fixed curve, polynomial values on difference sets, and
polynomial expansion. For a polynomial $F\in \F[U,V]$ of degree $D<p$
that is not a polynomial in one linear form, we prove
\[
 |F(P-P)|\gg_D\min\{|P|^{2/3},p\}
\]
whenever $P$ is not contained in a union of $D-1$ parallel lines.
For homogeneous forms and polynomials of degree two or three, we obtain
the same bound under weaker geometric hypotheses.
These results extend the usual distance bound to polynomial functions.
We compare the range $|P|\le p^{3/2}$ with the pinned-distance theorem of
Murphy, Petridis, Pham, Rudnev and Stevens~\cite{MPPRS} over general
fields. We also discuss their stronger results for the usual distance
over $\Fp$.
The section ends with examples showing the sharpness of the incidence
bound and the role of the two-parameter assumption. We do not optimize
the dependence on $t$ in the two main theorems.

\paragraph{Main ideas.}
We build on the method developed by Lewko \cite{Lewko}, which combines
polynomial interpolation in an additional parameter with estimates for
contact multiplicities. We retain this scheme and the final optimization
of the interpolation degree. For general algebraic families, the
additional steps concern shared components, singularities, and
exceptional factors in the contact estimate.
We first account for components that occur for infinitely many parameters.
We then bound the contribution from nonessential dual components
and partition the remaining points into classes with bounded numbers of
common neighbors (Lemmas~\ref{lem:nonessential} and~\ref{lem:twins}).
Persistent singularities are confined to a fixed curve of bounded degree
(Lemma~\ref{lem:W}). In the contact estimate, we show that every
exceptional factor defines a component of a family member or of the
zero set of the leading coefficient or discriminant in the family parameter
(Lemma~\ref{lem:norm}). Iterated tangent derivatives then relate contact
orders to the number of incident parameters (Lemma~\ref{lem:jets}).
Finally, the $K_{k,k}$-free hypothesis bounds the number of neighborhoods
containing almost all the points of a curve, which gives the reduction
to Boolean combinations of equations. The argument permits singular
and nonreduced members of the original family.

More precisely, for an integer $d$ with $1\le d<p/t$, interpolation gives
a polynomial of degree at most $d$ in the point variables $x$ and degree
$e=O_t(n/d)$ in one parameter variable.
The contact estimate bounds the total excess contact by $O_t(d^2)$,
while the structural reductions control the exceptional incidences.
Together they give
\[
 I(P,\CC)\ll_t \frac{mn}{d}+m+n+d^2.
\]
Balancing the first and last terms gives the exponent $2/3$.
In positive characteristic, imposing the restriction $td<p$ produces
the additional term $mn/p$.

Section~\ref{sec:preliminaries} fixes notation and states the preliminary
lemmas. We prove the structural lemmas in
Section~\ref{sec:structureproofs}, the contact bound in
Section~\ref{sec:norm}, and the interpolation and derivative estimates
in Section~\ref{sec:jets}. Sections~\ref{sec:assembly}
and~\ref{sec:boolean} prove Theorems~\ref{thm:pure}
and~\ref{thm:main}, respectively. Section~\ref{sec:applications}
contains the applications and sharpness examples. The appendix shows
the necessity of the degree restriction in the contact bound.

\section{Preliminary lemmas}\label{sec:preliminaries}

We write $V(h)$ for the affine zero set of a polynomial $h$. The letters
$\Lambda$ and $C$ denote curves in the point plane and parameter plane,
respectively. We write $q$ for a point and $y$ for a parameter; after a change of
parameter coordinates, $y=(a,b)$. The sets $P$ and $Q$ consist of points
and parameters, respectively, and $E$ denotes the edge set of a bipartite graph.
All implied constants depend only on $t$, unless another dependence is
indicated.
Fix an algebraic closure $\kk=\overline{\F}$. Passing from $\F$ to $\kk$
preserves incidences and Boolean truth values on the original finite sets.
We work over $\kk$ in Sections~\ref{sec:structureproofs}--\ref{sec:boolean}.
In $\family_L(f)$, the subscript specifies the field in which parameters
are chosen. We use $\CC$ for a finite collection of distinct components.
In particular, $\family_{\F}(f)\subseteq\family_\kk(f)$, so a bound for
the latter family applies to the given curves. By the definition of
$\family_{\F}(f)$, each such curve is a component of a fiber with parameter
in $\F^2$. We may then factor polynomials and change parameter coordinates
over $\kk$ without changing the incidences on the given finite sets.

Both main bounds are immediate if one of the finite sets is empty. If $p\le2t^2$, the trivial estimate
$I\le mn\le2t^2mn/p$ proves both main bounds. Thus, when proving those
bounds, we may assume
\begin{equation}\label{eq:largechar}
 p>2t^2\quad\hbox{or}\quad p=\infty.
\end{equation}
The lemmas below state their own characteristic assumptions when needed.

Lemmas~\ref{lem:bezout} and~\ref{lem:kernel} are standard algebraic
facts; we include short proofs for completeness. The remaining lemmas
are stated in the forms needed for the present argument. In particular,
Lemma~\ref{lem:interpolation} uses a standard dimension count, while
the contact and derivative estimates adapt Lewko's method \cite{Lewko}
to the curve families considered here.
The proofs of Lemmas~\ref{lem:bezout}--\ref{lem:W} are given in
Section~\ref{sec:structureproofs}, the proof of Lemma~\ref{lem:norm}
in Section~\ref{sec:norm}, and the proofs of
Lemmas~\ref{lem:interpolation}--\ref{lem:jets} in Section~\ref{sec:jets}.

\subsection{Algebraic facts}

We use the affine length form of B\'ezout's theorem; see
Fulton~\cite[Sections~3.3 and~5.3]{Fulton}. This form controls local
intersection multiplicities as well as the number of intersection points.

\begin{lemma}[Affine B\'ezout length]\label{lem:bezout}
Let $H_1,H_2\in L[X,Y]$ be nonzero coprime polynomials of degrees $a,b$ over
any field $L$. Then
\[
 \dim_L L[X,Y]/(H_1,H_2)\le ab.
\]
Over an algebraically closed field, the sum of their local intersection
multiplicities, and hence the number of their common points, is at most
$ab$.
\end{lemma}

The positive-characteristic assertion in the next lemma also follows
from the description of the kernel of the universal derivation in
\cite[Lemma~10.158.2, Tag~031W]{Stacks}.

\begin{lemma}[Kernel of a function-field derivation]\label{lem:kernel}
Let $M$ be a function field of one variable over an algebraically closed
field $\Omega$, and let $\Der$ be a nonzero $\Omega$-derivation of $M$.
In characteristic zero, $\ker \Der=\Omega$. In characteristic $p>0$,
$\ker \Der=M^p$.
\end{lemma}

\subsection{Components of a two-parameter family}\label{sec:structure}

We call a polynomial \emph{mixed} if it depends on both the point
variables $x=(x_1,x_2)$ and the parameter variables $y=(y_1,y_2)$.
In this subsection and the next, $f\in\kk[x_1,x_2,y_1,y_2]$ is a mixed
irreducible polynomial of total degree at most $t$. Set
\[
 \mathcal N_f=\{y:f(\,\cdot\,;y)\equiv0\},\qquad
 B_f=\{q:f(q;\,\cdot\,)\equiv0\},\qquad
 \Gamma_q=V(f(q;\,\cdot\,)).
\]
For $y\notin \mathcal N_f$, write $Z_y=Z_y(f)=V(f(\,\cdot\,;y))$.

\begin{lemma}[Common zeros and occurrence loci]\label{lem:occurrence}
If finitely many polynomials in $\kk[u,v]$ of degree at most $t$ have no
common nonconstant factor and are not all zero, their common zero set has
at most $t^2$ points. In particular,
\[
 |\mathcal N_f|\le t^2,\qquad |B_f|\le t^2.
\]
If $\Lambda$ is an irreducible component of some $Z_y$, the set
\[
 F_\Lambda=\{y\notin \mathcal N_f:\Lambda\subset Z_y\}
\]
is the complement of $\mathcal N_f$ in a set cut out by polynomials of degree at most
$t$. Its closure is a proper subset of the parameter plane. Indeed, equality would force an equation of $\Lambda$ to divide the
irreducible mixed polynomial $f$. If $F_\Lambda$ is finite, then
$|F_\Lambda|\le t^2$.
\end{lemma}

Call $\Lambda$ a \emph{base curve} if $F_\Lambda$ is infinite, and a
\emph{nonbase curve} otherwise. Thus, a base curve occurs as a component
for infinitely many parameter values; it need not be common to every
member of the family. Let $\Zess_y$ be the union of the nonbase
components of $Z_y$. For finite $P\subset\kk^2$ and
$Q\subset\kk^2\setminus \mathcal N_f$, put
\[
 \Iess(P,Q)=|\{(q,y)\in P\times Q:q\in\Zess_y\}|,
 \qquad r_Q^*(q)=|\{y\in Q:q\in\Zess_y\}|.
\]
This auxiliary incidence relation counts a point--parameter pair once,
even if the point belongs to several components of the same fiber.

For an irreducible parameter curve $C$, write $\eta_C$ for its generic
point and $\Omega_C=\overline{\kk(C)}$. A geometric component of
$f(x;\eta_C)=0$ is \emph{constant} if its equation is an
$\Omega_C^\times$ multiple of a polynomial in $\kk[x]$, and
\emph{nonconstant} otherwise. Define
\[
 E_C=\{q\in\kk^2:q\hbox{ lies on a nonconstant geometric component of }
                      f(x;\eta_C)=0\}.
\]
We say that $C$ is \emph{essential for $q$} if $q\in E_C$.
For $q\notin B_f$, any essential $C$ is an irreducible component of
$\Gamma_q$.
The generic polynomial $f(x;\eta_C)$ is nonzero, since $\mathcal N_f$ is finite.

\begin{lemma}[Constant points on a nonconstant curve]\label{lem:constantpoints}
Let $L\supset\kk$ be a field, and let $g\in L[x_1,x_2]$ be irreducible of
degree $a$, not a scalar multiple of a polynomial over $\kk$. Then
\[
 |V(g)\cap\kk^2|\le a^2.
\]
Consequently, $|E_C|\le t^3$ for every irreducible parameter curve $C$.
\end{lemma}

\begin{lemma}[Base curves through a point]\label{lem:base}
Every $q\in\kk^2\setminus B_f$ lies on at most $t^2$ base curves.
The constant components of $f(x;\eta_C)=0$ are exactly the irreducible
curves contained in every $Z_y$, $y\in C\setminus \mathcal N_f$; there are at most
$t$ of them.
\end{lemma}

\begin{lemma}[Nonessential dual components]\label{lem:nonessential}
Let $q\notin B_f$ and let $C$ be a component of $\Gamma_q$ which is not
essential for $q$. Then
\begin{equation}\label{eq:nonessential}
 |\{y\in C\setminus \mathcal N_f:q\in\Zess_y\}|
 \le t\deg C\le t^2.
\end{equation}
\end{lemma}

\begin{lemma}[Reduction and bounded codegrees]\label{lem:twins}
For every finite $P\subset\kk^2$, $|P|=m$, and $n$ distinct curves
$\CC\subset\family_\kk(f)$, there is a parameter set
$Q\subset\kk^2\setminus \mathcal N_f$, $|Q|\le n$, such that
\begin{equation}\label{eq:essentialreduction}
 I(P,\CC)\le t^2m+t^2n+t\Iess(P\setminus B_f,Q).
\end{equation}
Moreover:
\begin{enumerate}[label=(\roman*)]
\item For $q\notin B_f$,
\[
 r_Q^*(q)\le\sum_{C\text{ essential for }q}|Q\cap C|+t^3.
\]
\item $P\setminus B_f$ can be colored with at most $t^4$ colors so that
no two points in one color class belong to the same $E_C$.
\item For two distinct points in one such class,
\[
 |\{y\notin \mathcal N_f:q,q'\in\Zess_y\}|\le2t^3.
\]
Consequently, for every subset $P_0$ of a single class,
\begin{equation}\label{eq:pairbound}
 \Iess(P_0,Q)\le |Q|+2t^3\binom{|P_0|}{2}.
\end{equation}
\end{enumerate}
\end{lemma}

\subsection{A fixed exceptional curve}\label{sec:W}

The next lemma confines persistent singularities to a fixed curve in the
point plane. It does not require the specialized fibers to be reduced.

\begin{lemma}[A fixed exceptional curve]\label{lem:W}
Assume $p>t$ or $p=\infty$, and let $f$ be an irreducible mixed polynomial
as in Section~\ref{sec:structure}. There exist an index $j\in\{1,2\}$ with $g=f_{x_j}\ne0$ and a nonzero
polynomial $A_f\in\kk[x_1,x_2]$ of degree at most $4t(t-1)$ such that,
for every $q\notin V(A_f)\cup B_f$, the polynomial $g(q;\cdot)$ does not
vanish identically on any irreducible component of $\Gamma_q$.
This conclusion, with the same $A_f$, holds after any invertible linear
change of parameter coordinates.
For every finite $P\subset\kk^2$ with $|P|=m$ and every set $\CC$ of
$n$ distinct geometrically irreducible plane curves of degree at most $t$,
\begin{equation}\label{eq:Wcost}
 I(P\cap V(A_f),\CC)\le (\deg A_f)m+t(\deg A_f)n.
\end{equation}
\end{lemma}

\subsection{Contact multiplicity}

For a polynomial restricted to a smooth curve at a point $q$, we write
$\ord_q$ for its order in the local discrete valuation ring. We also use
$(u)_+=\max\{u,0\}$.
The following lemma bounds the total contact with members of a
one-parameter family after subtracting a bounded term at each point.

\begin{lemma}[Contact bound]\label{lem:norm}
Let $\Omega$ be algebraically closed, with $p=\operatorname{char}\Omega$
in positive characteristic and $p=\infty$ in characteristic zero. Let $\Phi\in\Omega[X,Y,B]$ have total degree at most
$t$ and $B$-degree $r\ge1$. Suppose
\begin{enumerate}[label=(\alph*)]
\item $\Disc_B\Phi\ne0$;
\item $\Phi(X,Y;c)\not\equiv0$ for every $c\in\Omega$.
\end{enumerate}
Let $F\in\Omega[X,Y]$ be nonzero and squarefree, of degree at most $d$,
where $1\le d<p/t$. Choose distinct points $q_i\in\Omega^2$ and
parameters $c_i\in\Omega$ such that
\[
 \Phi(q_i;c_i)=0,\qquad \nabla_{X,Y}\Phi(q_i;c_i)\ne0,
 \qquad\Phi(q_i;B)\not\equiv0.
\]
Let $\Lambda_i$ be the unique component of $\Phi(X,Y;c_i)=0$ through
$q_i$. Assume $F$ is not identically zero on $\Lambda_i$, and write
$\mu_i=\ord_{q_i}(F|_{\Lambda_i})$. Then
\begin{equation}\label{eq:normbound}
 \sum_i\epspart{\mu_i-c_*}\le rd(2t+d),
 \qquad c_*=1+(3r+1)t^2.
\end{equation}
\end{lemma}

\subsection{Interpolation and parameter roots}

Return to an irreducible mixed polynomial $f$ as above, and write its
parameter coordinates as $(a,b)$. For the finitely many points under
consideration, we choose these coordinates so that $f$ depends on $b$
and every irreducible component of $\Gamma_q$, for $q\notin B_f$,
dominates the $a$ axis. The existence of this choice is verified in
Section~\ref{sec:jets}.

\begin{lemma}[Interpolation]\label{lem:interpolation}
Let $Q\subset\kk^2\setminus\mathcal N_f$ be finite, and put $n_Q=|Q|$.
For $d\ge1$, put
\[
 e=\left\lfloor\frac{2tn_Q}{d+2}\right\rfloor.
\]
There is a nonzero polynomial $S(x,a)$, squarefree in $\kk[x_1,x_2,a]$,
with
\[
 \deg_x S\le d,\qquad \deg_aS\le e,
\]
such that $S(x,a_y)$ vanishes on every component of $Z_y$ for every
$y=(a_y,b_y)\in Q$.
\end{lemma}

\begin{lemma}[Derivative root count]\label{lem:jets}
Let $S,d,e,Q$ be as in Lemma~\ref{lem:interpolation}. Let $C$
be an irreducible component of $\Gamma_q$ dominating the $a$ axis. Write
its geometric generic point as $\eta_C=(z,\beta)\in\overline{\kk(z)}^2$, where
$z$ is transcendental. Suppose the gradient of $f(x;z,\beta)$ with respect to $x$ is nonzero
at $q$, and let $\Lambda$ be the unique component of $f(x;z,\beta)=0$
through $q$.
If $S(x,z)$ is not identically zero on $\Lambda$, put
$\mu=\ord_q(S(\,\cdot\,,z)|_\Lambda)$. If $td<p$, then
\begin{equation}\label{eq:jetbound}
 |Q\cap C|\le (\deg C)\bigl(e+(t-1)\mu\bigr)
 \le te+t(t-1)\mu.
\end{equation}
\end{lemma}

\section{Proofs of the structural lemmas
(Lemmas~\ref{lem:bezout}--\ref{lem:W})}\label{sec:structureproofs}

We first prove the two algebraic facts. We then prove the structural
lemmas for a two-parameter family.

\subsection{B\'ezout and derivations}

\begin{proof}[Proof of Lemma~\ref{lem:bezout}]
The assertion is immediate if either polynomial is a nonzero constant.
Write $\Pi_h=L[X,Y]_{\le h}$. For $h\ge a+b$, consider
\[
 J_h=H_1\Pi_{h-a}+H_2\Pi_{h-b}\subseteq (H_1,H_2)\cap\Pi_h.
\]
Coprimality and unique factorization give
$H_1\Pi_{h-a}\cap H_2\Pi_{h-b}=H_1H_2\Pi_{h-a-b}$. Therefore,
\[
 \begin{aligned}
 \dim\Pi_h-\dim J_h
 &=\binom{h+2}{2}-\binom{h-a+2}{2}
   -\binom{h-b+2}{2}+\binom{h-a-b+2}{2}\\
 &=ab.
 \end{aligned}
\]
Thus, the image of each $\Pi_h$ in the quotient has dimension at most
$ab$. These images form an increasing sequence whose union is the
quotient, so the quotient also has dimension at most $ab$. Over an
algebraically closed field, a finite-dimensional algebra is the product
of its local Artinian factors. Their lengths are the local intersection
multiplicities.
\end{proof}

\begin{proof}[Proof of Lemma~\ref{lem:kernel}]
In characteristic zero, if a nonconstant $u$ satisfies $\Der(u)=0$, then
$M/\Omega(u)$ is a finite separable extension. A derivation that vanishes on a field also vanishes on every separable
algebraic extension of that field. Hence, $\Der=0$, a contradiction.

In characteristic $p$, take any transcendental $u\in M$ and put
$e_M=[M:\Omega(u)]$. Frobenius and perfection of $\Omega$ give
$[M^p:\Omega(u^p)]=e_M$, whereas
$[M:\Omega(u^p)]=pe_M$. Thus, $[M:M^p]=p$. The kernel of $\Der$ is an
intermediate field containing $M^p$ and is proper because $\Der\ne0$.
Since this extension has prime degree, that intermediate field is $M^p$.
\end{proof}

\subsection{Base curves and essential components}

Throughout this subsection, $f$ is irreducible and depends on both
variable groups, as in Section~\ref{sec:structure}.

\begin{proof}[Proof of Lemma~\ref{lem:occurrence}]
Choose one nonzero polynomial $h$ in the given list. For each irreducible
factor of $h$, some polynomial in the list is not divisible by that factor.
Since $\kk$ is infinite, a generic $\kk$-linear combination of the
polynomials is coprime to $h$. The first assertion follows from B\'ezout
(if $h$ is a nonzero constant, the common zero set is empty). Applying
this assertion to the coefficients of $f$ in either variable group gives
the bounds for $\mathcal N_f$ and $B_f$:
a common coefficient factor would divide the irreducible mixed polynomial
$f$.

Let $g_\Lambda\in\kk[x]$ be an irreducible equation of $\Lambda$, of degree
$a$. Multiplication by $g_\Lambda$ identifies
$\kk[x]_{\le t-a}$ with a fixed linear subspace of $\kk[x]_{\le t}$.
The coefficient vector of $f(x;y)$ belongs to this subspace if and only
if it satisfies a fixed system of linear equations over $\kk$. Thus,
membership is expressed by polynomials in $y$ of degree at most $t$. These equations are valid over
every extension field. They cannot vanish identically in $y$, since that
would imply $g_\Lambda\mid f$ in $\kk[x,y]$. If their zero set is finite, the first assertion gives the bound $t^2$.
Conversely, removing the finite set $\mathcal N_f$ from a
positive-dimensional closed set over $\kk$ cannot leave a finite set.
\end{proof}

\begin{proof}[Proof of Lemma~\ref{lem:constantpoints}]
Choose a $\kk$-basis $e_1,\ldots,e_h$ of the span of the coefficients of
$g$, and write $g=\sum_j e_jg_j$ with $g_j\in\kk[x]$ of degree at most $a$.
A point of $\kk^2$ is a zero of $g$ exactly when it is a zero of every
$g_j$. A common nonconstant factor of the $g_j$ would, by irreducibility of
$g$, make $g$ a scalar multiple of a polynomial over $\kk$. Thus,
Lemma~\ref{lem:occurrence} gives at most $a^2$ zeros. There are at most $t$
nonconstant components of the generic member, each of degree at most $t$,
which proves the stated bound on $E_C$.
\end{proof}

\begin{proof}[Proof of Lemma~\ref{lem:base}]
The coefficient conditions of Lemma~\ref{lem:occurrence} show that a
constant irreducible equation divides $f(x;\eta_C)$ if and only if it
divides every specialization along $C$. There are at most $t$ such curves,
since they occur in any one nonzero member on $C$.

For each base curve $\Lambda$ through $q$, the closure of $F_\Lambda$
has an irreducible curve component $C_\Lambda$. All its nonzero members
contain $\Lambda$, so $C_\Lambda\subset\Gamma_q$. There are at most $t$ choices for $C_\Lambda$, since $\Gamma_q$ has
at most $t$ components. For each choice, the preceding characterization
gives at most $t$ possibilities for $\Lambda$.
\end{proof}

\begin{proof}[Proof of Lemma~\ref{lem:nonessential}]
Let $G_C(x)\in\kk[x]$ be the product, with their generic multiplicities, of
all constant geometric factors of $f(x;\eta_C)$, choosing representatives
normalized over $\kk$. We construct a quotient whose coefficients are polynomial in the parameters. The matrix of multiplication by $G_C$ from
$\kk[x]_{\le t-\deg G_C}$ to $\kk[x]_{\le t}$ has entries in $\kk$
and has a left inverse over $\kk$, since multiplication by $G_C$ is
injective. Apply that left inverse to the coefficients of
$f(x;y)$, and denote the resulting polynomial by $U_C(x;y)$. It has parameter
degree at most $t$, and
\[
 f(x;y)=G_C(x)U_C(x;y)\quad\hbox{for every }y\in C.
\]
Indeed, this equality holds at the generic point of $C$, so each coefficient
of the difference vanishes on $C$. At that generic point, the geometric
factors of $U_C$ are precisely the nonconstant factors. Since $q\notin E_C$,
$U_C(q;\eta_C)\ne0$.

Every factor of $G_C$ is a base curve by Lemma~\ref{lem:base}.
For $y\in C\setminus \mathcal N_f$, any nonbase component of $Z_y$ must therefore
be a component of $U_C(x;y)=0$. In particular, $q\in\Zess_y$ implies
$U_C(q;y)=0$. This polynomial has degree at most $t$ in $y$ and is not
identically zero on $C$. B\'ezout proves~\eqref{eq:nonessential}.
\end{proof}

\begin{proof}[Proof of Lemma~\ref{lem:twins}]
Points of $B_f$ contribute at most $t^2n$, and base curves contribute at
most $t^2m$ outside $B_f$. For each nonbase curve in $\CC$, choose a parameter in
$\kk^2\setminus\mathcal N_f$ for a fiber containing it, and let $Q$
be the set of chosen parameters. If $\CC\subset\family_{\F}(f)$, these
parameters may be chosen in $\F^2$. Each parameter corresponds to at most
$t$ curves in $\CC$, which gives
\eqref{eq:essentialreduction}. Part (i) follows from
Lemma~\ref{lem:nonessential}, with at most $t$ dual components per point.

On $P\setminus B_f$, join two points when they both belong to some
$E_C$. Each point belongs to at most $t$ such sets, each of size at most $t^3$.
Thus, this graph has maximum degree less than $t^4$. Greedy coloring gives (ii).

For (iii), remove the common factors from the equations of $\Gamma_q$
and $\Gamma_{q'}$. The remaining polynomials are coprime, so B\'ezout
bounds the number of intersection points outside the common curve
components by $t^2$. On each
common component $C$, at least one of $q,q'$ is not in $E_C$, by the
coloring. Lemma~\ref{lem:nonessential} bounds the number of parameters on $C$
for which $q,q'\in\Zess_y$ by $t^2$. There are at most $t$ such components. Finally, if
$s_y=|\{q\in P_0:q\in\Zess_y\}|$, then
$s_y\le1+\binom{s_y}{2}$. Sum this inequality and use (iii).
\end{proof}

\subsection{Persistent singularities}

\begin{proof}[Proof of Lemma~\ref{lem:W}]
At least one derivative $g=f_{x_j}$ is nonzero, since $f$ depends on
$x$ and has degree less than $p$. The polynomials $f,g$ are coprime, since $f$ is
irreducible and $\deg g<\deg f$. Write $r=\deg_y f\ge1$ and
$s=\deg_y g\ge0$.

Choose two nonparallel parameter directions along which the respective
degrees of $f$ and $g$ are $r$ and $s$. Such directions exist because the highest-degree homogeneous parts in
the parameter variables, with coefficients in $\kk[x]$, vanish in only
finitely many directions. For each chosen direction,
choose a nonzero linear form $u_i$ whose kernel is that direction, and
choose $v_i$ so that $(u_i,v_i)$ are linear coordinates. Thus, the lines
$u_i=\mathrm{constant}$ are parallel to the chosen direction. Form
\[
 R_i(x,u_i)=\Res_{v_i}^{(r,s)}(f,g),\qquad i=1,2.
\]
The superscript specifies the fixed formal degrees in the Sylvester
determinant. In particular, the determinant can be specialized without
changing its size. These resultants are
nonzero by coprimality over $\kk(x,u_i)$ and Gauss's lemma. If $s=0$,
$g$ depends only on $x$ and the convention is $R_i=g^r$. The degree bound
is
\[
 \deg_{x,u_i}R_i\le st+r(t-1)\le2t(t-1).
\]
Choose one nonzero coefficient $A_{f,i}(x)$ of each $R_i$ as a polynomial in
$u_i$, and put $A_f=A_{f,1}A_{f,2}$. Since both coefficients $A_{f,i}$ are nonzero,
$A_f$ is nonzero.

Suppose an irreducible parameter curve $C$ is a common zero component of
$f(q;y)$ and $g(q;y)$. The two linear forms $u_1,u_2$ are independent,
since their kernels are distinct. Consequently, a curve cannot be
contained in a level line of both forms, and $C$ dominates at least one
of the two $u_i$ axes. Over $\kk(u_i)$, the specialized polynomials then have a common factor
of positive degree in $v_i$, or one is zero and the other has such a
factor. Their fixed-degree resultant
is zero. This remains true when degrees drop: if $g(q;y)\equiv0$, the determinant is zero because it is homogeneous
of positive degree $r$ in the coefficients of $g$; if $s=0$, the
assertion is simply $g(q)^r=0$. If $f(q;y)\equiv0$ and $s>0$, homogeneity
of positive degree $s$ in its coefficients also makes the resultant zero.
Thus, $R_i(q,u_i)\equiv0$, so $A_{f,i}(q)=0$.

For~\eqref{eq:Wcost}, at most $\deg A_f$ of the distinct input curves can
be irreducible components of $V(A_f)$; they contribute at most $(\deg A_f)m$.
Every other input curve meets $V(A_f)$ in at most $t\deg A_f$ points.
If $A_f$ is constant, the exceptional set is empty.
\end{proof}

Outside $V(A_f)\cup B_f$, the gradient with respect to $x$ of the generic member along any
component of $\Gamma_q$ is nonzero at $q$. Hence, this member has a
unique component through $q$, and that component is smooth at $q$. Other components of that member may still
occur with multiplicity.

\begin{remark}\label{rem:persistent}
Genericity along a parameter curve alone does not guarantee smoothness
at the chosen point. In characteristic different from $2$, consider
\[
 f(x_1,x_2;a,b)=(x_2-b)^2-a(x_1-b)^2.
\]
For $q=(s,s)$, one has $f(q;a,b)=(s-b)^2(1-a)$. Along the essential
parameter curve $b=s$, the generic member consists of the two lines
$x_2-s=\pm\sqrt a\,(x_1-s)$, both passing through $q$. It is therefore
singular there. These points lie on the fixed diagonal $x_1=x_2$.
Lemma~\ref{lem:W} allows their incidences to be counted separately.
\end{remark}

\section{Proof of the contact bound
(Lemma~\ref{lem:norm})}\label{sec:norm}

We now prove the contact bound. We first identify the factors for which
the resultant argument does not apply. We then estimate their
contribution separately and use B\'ezout's theorem for the remaining
factors.

\begin{proof}[Proof of Lemma~\ref{lem:norm}]
Write $\ell_\Phi=\operatorname{lc}_B\Phi$ and
$\Delta_\Phi=\Disc_B\Phi$. Define the derivation in the point variables by
$\Der_\Phi=\Phi_Y\partial_X-\Phi_X\partial_Y$ and, for a polynomial
$H\in\Omega[X,Y]$, put
\[
 \mathcal R_\Phi(H)=\Res_B^{(r,r)}(\Phi,\Der_\Phi H).
\]
We regard both polynomials as having degree $r$ in $B$, inserting zero
coefficients when necessary. With this convention, the resultant
identities remain valid after specialization, including when the degree
drops. If $\mathcal R_\Phi(H)\ne0$, homogeneity of the resultant gives
\begin{equation}\label{eq:normdegree}
 \deg \mathcal R_\Phi(H)\le r(2t+\deg H).
\end{equation}
Since $\Der_\Phi(gh)\equiv h\Der_\Phi g\pmod g$, the same property gives
\begin{equation}\label{eq:normproduct}
 \mathcal R_\Phi(gh)\equiv h^r\mathcal R_\Phi(g)\pmod g.
\end{equation}

\smallskip
Let $g$ be an irreducible factor of degree $a<p/t$ such that
$g\mid \mathcal R_\Phi(g)$. We treat the factors of
$\ell_\Phi\Delta_\Phi$ separately in the last paragraph of this proof. Assume, therefore, that
$g\nmid\ell_\Phi\Delta_\Phi$, and let $M=\Omega(V(g))$ be the function
field of $V(g)$. In $M[B]$, the restriction of $\Phi$ has degree $r$ and
distinct roots.
If the restriction of $\Der_\Phi g$ is nonzero and has degree $s\le r$, the
fixed-degree resultant satisfies
\[
 \Res_B^{(r,r)}(\Phi,\Der_\Phi g)
 =\ell_\Phi^{\,r-s}\Res_B(\Phi,\Der_\Phi g),
\]
up to a sign depending on the resultant convention. Here $\ell_\Phi$ is
nonzero in $M$, so the two polynomials have a common root $b$.
If $\Der_\Phi g$ is zero in $M[B]$, choose any root $b$ of $\Phi$ instead.
In either case $(\Der_\Phi g)(b)=0$, and $M_1=M(b)$ is a finite separable
extension of $M$.
The derivation
\[
 \Der_g=g_Y\partial_X-g_X\partial_Y
\]
induces a nonzero derivation of $M$. Indeed, $\deg g<p$ implies that
at least one partial derivative of $g$ is nonzero. Its degree is smaller
than $\deg g$, so it is nonzero in $M$. This derivation extends uniquely
to $M_1$. Differentiating $\Phi(X,Y;b)=0$ gives
\[
 0=-(\Der_\Phi g)(b)+\Phi_B(b)\Der_g(b),
\]
so $\Der_g(b)=0$, since $\Phi_B(b)\ne0$.

In characteristic zero, Lemma~\ref{lem:kernel} yields $b\in\Omega$.
In positive characteristic it yields $b=\beta^p$ for some $\beta\in M_1$.
If $b\notin\Omega$, then $b$ is transcendental over the algebraically
closed field $\Omega$. Thus, evaluation at $B=b$ identifies $\Omega(B)$
with $\Omega(b)$. Absolute irreducibility of $g$ implies that it remains
irreducible over $\Omega(B)$. Moreover, $g$ and $\Phi(X,Y;B)$ are coprime
there: divisibility would force every $B$-coefficient of $\Phi$ to vanish
modulo $g$, contrary to $g\nmid\ell_\Phi$. Affine B\'ezout gives
\[
 \dim_{\Omega(B)}\Omega(B)[X,Y]/(g,\Phi(X,Y;B))\le ta.
\]
Evaluation at $B=b$ and at the coordinate functions $X,Y$ maps this
algebra into $M_1$. Its image is a finite-dimensional domain over
$\Omega(b)$, hence a field. It contains the coordinate functions $X,Y$
and $\Omega(b)$, so it equals $M_1$.
Therefore,
\[
 [M_1:\Omega(b)]\le ta<p.
\]
But the intermediate extension
$\Omega(\beta)/\Omega(\beta^p)=\Omega(\beta)/\Omega(b)$ has degree $p$,
a contradiction. Thus, in both characteristics, $b=c\in\Omega$ and
$g\mid\Phi(X,Y;c)$. We have proved that every factor under consideration
defines a component of a member $\Phi(X,Y;c)=0$ with $c\in\Omega$, of
$V(\ell_\Phi)$, or of $V(\Delta_\Phi)$.

\smallskip

For coprime plane
polynomials $H_1,H_2$, write $I_q(H_1,H_2)$ for their local intersection
multiplicity at $q$. In the local discrete
valuation ring of the smooth component $\Lambda_i$ at $q_i$, polynomial
division gives $\Phi(X,Y;B)=(B-c_i)\widehat\Phi_i(B)$. To apply the
resultant identity without a degree assumption after restriction, homogenize
the two coefficient vectors to binary forms of formal degree $r$ using
a second variable $T$. The superscript $h$ denotes homogenization,
and $\Res_{r,r}$ denotes the homogeneous resultant of the two degree-$r$
binary forms, equivalently the fixed-degree resultant used above.
In that local ring,
\[
 \Phi^h(B,T)=(B-c_iT)\widehat\Phi_i^{\,h}(B,T),\qquad
 \Res_{r,r}(\Phi^h,(\Der_\Phi H)^h)
 =(\Der_\Phi H)(c_i)\Res_{r-1,r}(\widehat\Phi_i^{\,h},(\Der_\Phi H)^h),
\]
up to a sign. This identity remains valid even
when leading coefficients vanish or a restricted form becomes zero.
Thus, $\mathcal R_\Phi(H)|_{\Lambda_i}$ is divisible by $(\Der_\Phi H)(c_i)|_{\Lambda_i}$.
To justify the order estimate, write the specialized polynomial locally
as $\Phi(X,Y;c_i)=h_i g_i$, where $g_i=0$ defines $\Lambda_i$ and $h_i$
is a unit at $q_i$. The hypothesis on $\nabla_{X,Y}\Phi(q_i;c_i)$ ensures that
$g_i$ occurs with multiplicity one and is smooth there. On $\Lambda_i$,
the derivation
\[
 \Der_{\Phi,c_i}=\Phi_Y(X,Y;c_i)\partial_X-\Phi_X(X,Y;c_i)\partial_Y
\]
equals $h_i(g_{i,Y}\partial_X-g_{i,X}\partial_Y)$. It preserves the
local ring of $\Lambda_i$ and is nonzero at $q_i$. In a local
uniformizer $u$ it therefore has the form $w(u)d/du$, with $w(0)\ne0$.
Since $(\Der_\Phi H)(c_i)|_{\Lambda_i}=\Der_{\Phi,c_i}(H|_{\Lambda_i})$, its order is at
least $\ord_{q_i}(H|_{\Lambda_i})-1$.
If $H,\mathcal R_\Phi(H)$ are coprime, the local plane
intersection algebra maps onto the quotient on $\Lambda_i$. The latter
has length $\min\{\ord(H|_{\Lambda_i}),\ord(\mathcal R_\Phi(H)|_{\Lambda_i})\}$,
with order $\infty$ for the zero restriction. Hence,
\begin{equation}\label{eq:localcontact}
 I_{q_i}(H,\mathcal R_\Phi(H))\ge
 \epspart{\ord_{q_i}(H|_{\Lambda_i})-1}.
\end{equation}
This also applies when the restriction of the resultant to $\Lambda_i$
is zero; the length of the quotient is then the order of $H$ itself.

\smallskip

Factor $F$ over $\Omega$. Let $F_0$ be the product of those factors that
define a component of some member $\Phi(X,Y;c)=0$, of $V(\ell_\Phi)$, or of
$V(\Delta_\Phi)$, and set $F_1=F/F_0$.
Combining equation~\eqref{eq:normproduct} and squarefreeness
implies $\gcd(F_1,\mathcal R_\Phi(F_1))=1$ if $F_1$ is nonconstant. In this case
$\mathcal R_\Phi(F_1)\ne0$, since otherwise its gcd with the nonconstant polynomial
$F_1$ would be $F_1$. Applying
\eqref{eq:localcontact} and B\'ezout gives
\begin{equation}\label{eq:residualcontact}
 \sum_i\epspart{\ord_{q_i}(F_1|_{\Lambda_i})-1}
 \le rd(2t+d).
\end{equation}
If $F_1$ is constant, the left-hand side is zero.

At a marked point $q_i$, the nonzero polynomial $\Phi(q_i;B)$ has at most
$r$ distinct roots. The factors of $F_0$ through $q_i$ that occur in
members of the family therefore lie in at most $r$ nonzero members, whose degrees sum to at most
$rt$. Since the discriminant is homogeneous of degree $2r-2$ in the
$B$-coefficients, each of degree at most $t$, we have
$\deg\Delta_\Phi\le(2r-2)t\le2rt$. The leading-coefficient and
discriminant factors therefore have total degree at most $t+2rt$.
Thus, the sum of degrees of all distinct factors of $F_0$
through $q_i$ is at most $(3r+1)t$. None is $\Lambda_i$, by the hypothesis
on $F$. B\'ezout consequently gives
\[
 \ord_{q_i}(F_0|_{\Lambda_i})\le(3r+1)t^2=c_*-1.
\]
Since orders add under multiplication, this estimate
and~\eqref{eq:residualcontact} give~\eqref{eq:normbound}.
\end{proof}

\begin{remark}[Squarefreeness over an imperfect intermediate field]\label{rem:squarefree}
We will apply the lemma after extending $L=\kk(z)$ to its algebraic
closure. A squarefree polynomial over $L$ of degree less than $p$ stays
squarefree over that closure. Indeed, each irreducible factor has a nonzero
partial derivative, coprime to that factor. These coprimality relations
persist under extension of fields: the finite-dimensional quotient in
Lemma~\ref{lem:bezout} stays finite-dimensional after scalar extension,
whereas a common curve factor would give an infinite-dimensional
quotient. A repeated geometric factor
would divide the polynomial and all its partial derivatives. Distinct
irreducible factors remain coprime under extension as well. The same
argument applies in characteristic zero without a degree restriction.
\end{remark}

\begin{remark}[Exceptional factors]
In characteristic zero, take $\Phi=Y-BX-B^2$ and $F=X^2+4Y$.
Each line $Y=cX+c^2$ is tangent to $F=0$ at $(-2c,-c^2)$, with contact
order $2$. Since we may choose arbitrarily many such points, a bound for
$\sum(\mu_i-1)_+$ requires a separate treatment of exceptional factors.
Here $F$ is exactly the $B$-discriminant; the resultant
$\mathcal R_\Phi(F)$ is a nonzero scalar multiple of $F$. A different issue occurs for
$\Phi=XB^2+YB+1$ and $F=X$: the ordinary resultant of $\Phi$ with
$\Der_\Phi F=B$ is $1$, while the fixed-degree $(2,2)$ resultant is $X$.
This factor comes from the leading coefficient. The removal of both
discriminant and leading-coefficient factors in Lemma~\ref{lem:norm}
is therefore necessary in this argument.
\end{remark}

\section{Proofs of the interpolation and derivative estimates
\texorpdfstring{\\}{ }(Lemmas~\ref{lem:interpolation}--\ref{lem:jets})}\label{sec:jets}

We now return to an irreducible polynomial $f$ over $\kk$ that depends on
both groups of variables. Make a linear change
of parameter coordinates and write them as $(a,b)$, so that $f$ depends on
$b$ and every irreducible component of $\Gamma_q$, for the finitely many
points $q\notin B_f$ under consideration, dominates the $a$ axis.
Such a change exists over the infinite field $\kk$: a curve fails to
dominate a linear projection only if it is a line parallel to its kernel,
and there are only finitely many relevant components. A generic kernel
also makes the $b$ degree of $f$ equal to its positive total parameter
degree. All total degrees and incidences are preserved.

\begin{proof}[Proof of Lemma~\ref{lem:interpolation}]
Consider the polynomials $S_0(x,a)$ satisfying the two degree bounds in
the statement. They form a vector space of dimension $(e+1)\binom{d+2}{2}$. For each $y\in Q$, impose
$f(x;y)\mid S_0(x,a_y)$. If the nonzero fiber polynomial has degree $h$,
the number of linear conditions is
\[
 \binom{d+2}{2}-\binom{d-h+2}{2}\le h(d+1)\le t(d+1)
\]
when $h\le d$, and is $\binom{d+2}{2}$ when $h>d$. In the latter case
$d\le h-1\le t-1$, so again $\binom{d+2}{2}\le t(d+1)$.
A nonzero constant fiber imposes zero conditions. Since
$(e+1)(d+2)>2tn_Q$, the total number of conditions is smaller than the
number of unknowns. Take a nonzero solution and replace it by its
squarefree part in $\kk[x,a]$. Neither degree bound increases.

The vanishing property is preserved after specialization: an irreducible
fiber equation dividing a specialized product divides at least one
specialized factor, or a factor specializes to the zero polynomial.
Removing repeated factors retains that vanishing factor. This proves
the claim even if the specialization of the squarefree part has repeated
factors or is zero.
\end{proof}

\begin{proof}[Proof of Lemma~\ref{lem:jets}]
Consider the derivation in the point variables
\[
 \Der_f=f_{x_2}\partial_{x_1}-f_{x_1}\partial_{x_2}.
\]
For every parameter $y$, the specialized derivation $\Der_{f,y}$
preserves the ideal of each reduced irreducible component of the fiber,
even when the fiber is not reduced.
Indeed, if $f(\,\cdot\,;y)=g^s h$, then cancellation of the two terms containing $s$
gives
\[
 \Der_{f,y}(g)=g^s(h_{x_2}g_{x_1}-h_{x_1}g_{x_2})\in(g).
\]
This polynomial identity is valid after every extension of the
coefficient field. It therefore also shows that the derivation at the
generic parameter $(z,\beta)$ preserves the ideal of $\Lambda$.
For each input parameter $y\in Q$, every iterate $\Der_{f,y}^jS(x,a_y)$
therefore vanishes on each component of $Z_y$. In
particular, every $y\in Q\cap C$ is a zero of
\[
 \Theta_j(a,b)=(\Der_f^jS)(q,a,b),\qquad
 \deg_{a,b}\Theta_j\le e+j(t-1).
\]
The degree estimate holds because differentiation in $x$ does not
increase parameter degree and the coefficients of $\Der_f$ have parameter
degree at most $t-1$.

B\'ezout gives $\mu\le td<p$. Work over $\Omega=\overline{\kk(z)}$.
In the following local calculation, specialize $a=z$ and $b=\beta$,
and restrict $S$ and $\Der_f$ to $\Lambda$.
If $f_{x_2}(q;z,\beta)\ne0$, take $u=x_1-q_1$; otherwise take
$u=x_2-q_2$. Smoothness shows that $u$ is a local uniformizer on
$\Lambda$, and $\Der_f(u)$ is a unit. The induced derivation of the local
ring extends continuously to its completion $\Omega[[u]]$ and has
the form $w(u)d/du$, where $w(u)=\Der_f(u)$ and $w(0)\ne0$. Thus,
\[
 S=c u^\mu+O(u^{\mu+1}),\quad c\ne0,\qquad
 \Der_f=w(u)\frac{d}{du},\quad w(0)\ne0.
\]
Consequently,
\[
 (\Der_f^\mu S)(q,z,\beta)=\mu!c\,w(0)^\mu\ne0.
\]
For $\mu=0$, this says $S(q,z)\ne0$. Hence, $\Theta_\mu$ is not
identically zero on $C$. Its zeros on $C$ include $Q\cap C$, so parameter
plane B\'ezout gives~\eqref{eq:jetbound}.
\end{proof}

\section{Proof of Theorem~\ref{thm:pure}}\label{sec:assembly}

We now combine interpolation with the contact bound and the structural
reductions of Section~\ref{sec:preliminaries}.

\begin{proof}[Proof of Theorem~\ref{thm:pure}]
\smallskip

We first prove the assertion for an irreducible mixed $f$. Let
$W=V(A_f)$ be the exceptional curve from Lemma~\ref{lem:W}, and put
$P'=P\setminus(W\cup B_f)$. First apply~\eqref{eq:Wcost} to the
points on $W$, and then apply Lemma~\ref{lem:twins} to
$P\setminus W$. The latter lemma bounds the incidences at $B_f$ and on the base
curves. Since $\deg A_f\le4t^2$, there is a parameter set $Q$ with
$|Q|=n_Q\le n$ such that
\begin{equation}\label{eq:initialcost}
 I(P,\CC)\le5t^2m+5t^3n+t\Iess(P',Q).
\end{equation}
If $n_Q=0$, this proves the required bound. Choose parameter coordinates
as in Section~\ref{sec:jets} for the finite set $P'$. This choice may
depend on $P'$ and is made before constructing the interpolating
polynomial. The geometric conclusion of Lemma~\ref{lem:W} is
invariant under this change.

Using the original bounds $m,n\ge1$, we choose the same interpolation
degrees for all color classes:
\begin{equation}\label{eq:degrees}
 d_0=(mn)^{1/3},\qquad
 d=\begin{cases}
 \lfloor d_0\rfloor,&p=\infty,\\
 \min\{\lfloor d_0\rfloor,\lceil p/t\rceil-1\},&p<\infty,
 \end{cases}
 \qquad e=\left\lfloor\frac{2tn_Q}{d+2}\right\rfloor.
\end{equation}
In positive characteristic, $p>t$ gives $\lceil p/t\rceil-1\ge1$.
Then $1\le d\le d_0$, $td<p$, and
\begin{equation}\label{eq:degreeopt}
 \frac1{d+2}\le\frac{1}{d_0}+\frac tp,
 \qquad me\le2t\left((mn)^{2/3}+\frac{t mn}{p}\right).
\end{equation}
Let $S$ be the polynomial given by Lemma~\ref{lem:interpolation}.

\smallskip

Color $P'$ as in Lemma~\ref{lem:twins}, and fix a color class $P_i$.
Let $z$ be transcendental over $\kk$, and put $\Omega=\overline{\kk(z)}$.
For each point $q\in P_i$ with an essential dual component, choose one
such component $C_q$ for which $|Q\cap C_q|$ is maximal. Every chosen
component dominates the $a$ axis. Choose an embedding
$\kk(C_q)\hookrightarrow\Omega$ sending $a$ to $z$, and write its
generic point as $(z,\beta_q)$. The fiber is smooth at $q$ by
Lemma~\ref{lem:W}. Let $\Lambda_q$ be its unique component through $q$. This component
is nonconstant, since $C_q$ is essential for $q$ and the fiber is
smooth at $q$. Put
\[
 P_i^{\mathrm{exc}}=\{q\in P_i:C_q\hbox{ was chosen and }S(x,z)|_{\Lambda_q}\equiv0\}.
\]
For $q\in P_i^{\mathrm{exc}}$, a defining polynomial of $\Lambda_q$
is a geometric irreducible factor of $S(x,z)$. Its degree is at most
$t$, since $\Lambda_q$ is a component of a member of the family.
Let $g_j\in\Omega[x]$ be pairwise nonassociate irreducible equations of
the distinct curves obtained in this way, and write $d_j=\deg g_j\le t$.
None is a scalar multiple of a polynomial over $\kk$. Their product
divides $S(x,z)$, so $\sum_jd_j\le\deg_xS(x,z)\le d$.
Thus, Lemma~\ref{lem:constantpoints} gives
$|P_i^{\mathrm{exc}}|\le\sum_j d_j^2\le td$. Using~\eqref{eq:pairbound}, we obtain
\begin{equation}\label{eq:Zsize}
 |P_i^{\mathrm{exc}}|\le td,
 \qquad \Iess(P_i^{\mathrm{exc}},Q)\le n_Q+t^5d^2.
\end{equation}

\smallskip

For each remaining selected point, let
$\mu_q=\ord_q(S(\,\cdot\,,z)|_{\Lambda_q})$. Apply Lemma~\ref{lem:norm} over
$\Omega$ to
\[
 \Phi(x;B)=f(x;z,B),\qquad F(x)=S(x,z).
\]
We verify the hypotheses of Lemma~\ref{lem:norm}:
\begin{itemize}[leftmargin=20pt]
\item As a polynomial in $x,b$ over $\kk[a]$, $f$ is primitive because
it is irreducible and depends on $x$. By Gauss's lemma, it remains
irreducible in $\kk(z)[x,b]$. Since it also depends on $b$, a second
application of Gauss's lemma gives irreducibility in $\kk(z,x)[b]$.
Its degree $r$ in $b$ satisfies $1\le r\le t<p$. Hence, it is
separable and its discriminant is nonzero.
\item Its coefficients in $x$, viewed in $\kk(z)[b]$, have gcd $1$:
a nonconstant common factor would contradict that irreducibility and
the dependence on $x$. This gcd remains $1$ over $\Omega[b]$.
Thus, $\Phi(x;c)$ is not identically zero for any $c\in\Omega$.
\item $q\notin B_f$ ensures $\Phi(q;B)\not\equiv0$.
Lemma~\ref{lem:W} gives $\nabla_x\Phi(q;\beta_q)\ne0$.
\item $S(x,z)$ is nonzero and squarefree over $\kk(z)$, by localization
of its squarefree factorization in $\kk[x,a]$. Its degree in $x$ is
less than $p$, so it is geometrically squarefree by
Remark~\ref{rem:squarefree}. If it is constant in $x$, then every
selected point is nonexceptional and has contact order zero. In this
case, Lemma~\ref{lem:jets} suffices. Otherwise,
Lemma~\ref{lem:norm} applies.
\end{itemize}
Therefore,
\begin{equation}\label{eq:sumcontact}
 \sum_{q\in P_i\setminus P_i^{\mathrm{exc}}\text{ selected}}\mu_q
 \le c_*|P_i|+td(2t+d)
 \le5t^3|P_i|+3t^2d^2,
\end{equation}
where $c_*=1+(3r+1)t^2\le5t^3$ and $d\ge1$.

\smallskip

Using Lemmas~\ref{lem:twins} and~\ref{lem:jets}, each nonexceptional
selected point satisfies
\[
 r_Q^*(q)\le t|Q\cap C_q|+t^3
 \le t^2e+t^2(t-1)\mu_q+t^3.
\]
Points with no essential component have $r_Q^*(q)\le t^3$. Combining these
bounds with~\eqref{eq:Zsize}--\eqref{eq:sumcontact}, we obtain
\[
 \Iess(P_i,Q)\le t^2e|P_i|+6t^6|P_i|+n_Q+4t^5d^2.
\]
There are at most $t^4$ color classes and their total size is at most
$m$. Therefore,
\[
 \Iess(P',Q)\le t^2em+6t^6m+t^4n+4t^9d^2.
\]
Substituting in~\eqref{eq:initialcost} gives
\[
 I(P,\CC)\le t^3em+11t^7m+6t^5n+4t^{10}d^2.
\]
By~\eqref{eq:degreeopt} and $d^2\le(mn)^{2/3}$, the right side is
at most
\[
 11t^{10}\left((mn)^{2/3}+m+n+\frac{mn}{p}\right).
\]
This proves the assertion for irreducible mixed $f$, with the stated
dependence on $t$.

\smallskip

For a general nonzero $f$, factor it over $\kk$ into at most $t$ distinct
irreducible factors. Every component of a nonzero fiber of $f$ is a
component of a nonzero fiber of at least one factor; assign each input
curve to one such factor. A factor depending only on the parameters gives
no curve in a nonzero fiber. A factor depending only on $x$ gives one
fixed irreducible curve and contributes at most $m$ incidences. For each mixed
factor, apply the bound just proved, using $n$ as an upper bound for
the number of assigned curves. Summing over at most $t$ factors proves the bound with
$C_t\le C_{\mathrm{abs}}t^{11}$ for an absolute constant $C_{\mathrm{abs}}$. The zero polynomial
has no nonzero fibers, and constant polynomials are immediate. The
small-characteristic estimate before~\eqref{eq:largechar} is also
covered by this choice of constant. This completes the proof of
Theorem~\ref{thm:pure}.
\end{proof}

\section{Proof of Theorem~\ref{thm:main}}\label{sec:boolean}

We pass from individual polynomial families to Boolean combinations of
equations. The $K_{k,k}$-free hypothesis bounds how many neighborhoods
can contain almost all points of a rich curve.

\begin{proof}[Proof of Theorem~\ref{thm:main}]
The case $k=1$ has no edges, so assume $k\ge2$. Let
$f_1,\ldots,f_s$, where $s\le t$, and $\Psi$ describe the graph as in
Section~1. We call each equation $f_i(x;y)=0$ an \emph{atom} of this
description. Let $\CC_Q$ be the set of distinct irreducible components
of all nonzero polynomials $f_i(\,\cdot\,;y)$, with $i\le s$ and
$y\in Q$. Then $|\CC_Q|\le t^2n$.

For $y\in Q$, define
\[
 \varepsilon(y)=
 \Psi\bigl([f_1(\,\cdot\,;y)\equiv0],\ldots,
           [f_s(\,\cdot\,;y)\equiv0]\bigr),
 \qquad Q_j=\{y\in Q:\varepsilon(y)=j\}\quad(j=0,1).
\]
Thus, $\varepsilon(y)$ is the adjacency value away from the zero sets
of the nonzero atom fibers. This divides the parameters into two classes,
even when the set of identically zero atoms varies within a class.

For $\Lambda\in\CC_Q$ and $y\in Q$, each atom is either identically zero
on $\Lambda$ or has at most $t\deg \Lambda\le t^2$ zeros there. Consequently, the
truth values of all atoms, and hence the adjacency value, are constant away
from at most $t^3$ points of $\Lambda$. Call $(y,\Lambda)$ \emph{full} if this
generic value is $1$, and \emph{sparse} otherwise. A full pair omits at
most $t^3$ points of $P\cap \Lambda$ from the neighborhood of $y$; a sparse
pair includes at most $t^3$ such points.

Call a curve $\Lambda$ \emph{rich} if $|P\cap \Lambda|\ge k(t^3+1)$. If $k$
distinct parameters are full on a rich curve, deleting their at most
$kt^3$ exceptional points leaves at least $k$ common neighbors. This
contradicts the $K_{k,k}$-free hypothesis. Thus, for every rich curve $\Lambda$, at most $k-1$ parameters
$y\in Q$ give a full pair $(y,\Lambda)$.

We first count the edges with parameter in $Q_0$. Every such edge lies
on a component of a nonzero atom fiber for its parameter, and there are
at most $t^2$ such components per parameter. Sparse pairs therefore
contribute at most $t^5|Q_0|$. Full pairs on curves that are not rich
contribute at most $k(t^3+1)t^2|Q_0|$, while full pairs on rich curves
contribute at most $(k-1)I(P,\CC_Q)$. Hence,
\begin{equation}\label{eq:booleanzero}
 |E\cap(P\times Q_0)|
 \le (k-1)I(P,\CC_Q)+(2k+1)t^5|Q_0|.
\end{equation}

Next consider $Q_1$. If $|Q_1|<k$, there are at most $(k-1)m$ edges
to count. Otherwise, choose $k$ parameters from $Q_1$. Every point
outside the zero sets of their nonzero atom fibers is adjacent to all
$k$ chosen parameters. There are therefore at most $k-1$ such points of $P$.
The remaining points are covered by at most $kt^2$ irreducible curves.
These are components of atom fibers of the chosen parameters, so they
belong to $\CC_Q$ and the preceding rich-curve bound applies to them.

On the covering curves, sparse pairs contribute at most $kt^5|Q_1|$.
Full pairs on curves that are not rich contribute at most
$k^2t^2(t^3+1)|Q_1|$. On rich covering curves, full pairs contribute
at most $k(k-1)t^2m$. The uncovered points contribute at most
$(k-1)|Q_1|$. An edge may be counted on more than one covering curve, which is
harmless for this upper bound. Since $k\ge2$
and $t\ge1$, these estimates, including the case $|Q_1|<k$, give
\begin{equation}\label{eq:booleanone}
 |E\cap(P\times Q_1)|
 \le k^2t^2m+3k^2t^5|Q_1|.
\end{equation}
Combining~\eqref{eq:booleanzero} and~\eqref{eq:booleanone}, we obtain
\[
 |E|\le (k-1)I(P,\CC_Q)+3k^2t^5(m+n).
\]
The argument also applies to constant Boolean formulas and empty lists
of nonzero atoms.

Finally, assign each curve of $\CC_Q$ to one atom in whose nonzero
fiber it occurs, and let $\CC_i$ be the set assigned to $f_i$. Each
$\CC_i$ has at most $tn$ elements. Apply Theorem~\ref{thm:pure}
separately to the original polynomials $f_i$. Its bound
$C_t\le C_{\mathrm{abs}} t^{11}$ gives
\begin{align*}
 I(P,\CC_Q)
 &\le C_{\mathrm{abs}} t^{11}\sum_{i=1}^s
 \left((m|\CC_i|)^{2/3}+m+|\CC_i|
                         +\frac{m|\CC_i|}{p}\right)\\
 &\le C_{\mathrm{abs}} t^{13}
 \left((mn)^{2/3}+m+n+\frac{mn}{p}\right).
\end{align*}
Since we apply the theorem to the original atoms, identically zero
specializations introduce no additional curves. Substitution
in the preceding edge estimate proves Theorem~\ref{thm:main}, with
$C(t,k)\le C_{\mathrm{abs}} k^2t^{13}$ for an absolute constant $C_{\mathrm{abs}}$.
\end{proof}

\section{Applications and sharpness}\label{sec:applications}

The main applications concern non-Cartesian polynomials on products
of planar sets (Corollary~\ref{cor:noncartesian}) and polynomial values
on difference sets (Section~\ref{sec:polynomialvalues}). We also derive
bounds for rich components, translates of a fixed curve, and polynomial
expansion, and conclude with sharpness examples.

\subsection{Cartesian products of two-dimensional sets}
\label{sec:cartesian}

Nassajian Mojarrad, Pham, Valculescu and de Zeeuw
\cite[Theorem~1.3]{NMPVdZ} studied polynomial zero sets on Cartesian
products of two planar sets. They proved that, for a non-Cartesian polynomial
$f\in\mathbb C[x_1,x_2,y_1,y_2]$ of degree at most $t$ and finite
sets $P,Q\subset\mathbb C^2$ of sizes $m,n$,
\[
 |V(f)\cap(P\times Q)|
 \ll_{t,\varepsilon}m^{2/3+\varepsilon}n^{2/3}+m+n.
\]
When both point sets are real, their bound holds with $\varepsilon=0$.
We first recall the non-Cartesian hypothesis.

Write $\kk=\overline{\F}$. A polynomial $f\in \F[x_1,x_2,y_1,y_2]$
is \emph{geometrically Cartesian} if there are irreducible curves
$\Lambda,C\subset\kk^2$ such that
$\Lambda\times C\subset V(f)$.
Equivalently, there are nonconstant irreducible polynomials
$g\in\kk[x_1,x_2]$ and $h\in\kk[y_1,y_2]$ such that
\begin{equation}\label{eq:cartesianform}
 f(x;y)=g(x)H(x;y)+h(y)J(x;y)
\end{equation}
for polynomials $H,J$ over $\kk$.
The displayed form implies the product inclusion.
Conversely, divide $f$ by $g$ in the point variables, treating the
parameter variables as coefficients, and use a monomial order that
respects total degree. Write $f=gH+R_g$, where no monomial of $R_g$ in
the point variables is divisible by the leading monomial of $g$.
For $y\in V(h)$, the product inclusion implies that $g$ divides
$R_g(\,\cdot\,;y)$, so the remainder property forces this specialization
to be zero. Thus, every coefficient of $R_g$ vanishes on $V(h)$ and is
divisible by $h$, proving~\eqref{eq:cartesianform}.
This is equivalent to the definition in \cite{NMPVdZ}, where $g$
and $h$ may be reducible: one may replace each by an irreducible factor.

\begin{corollary}[Non-Cartesian polynomials]\label{cor:noncartesian}
Let $\F$ be a field of characteristic $p$, with the usual convention
in characteristic zero. Suppose that $f\in \F[x_1,x_2,y_1,y_2]$ has
degree at most $t$ and is not geometrically Cartesian. Then, for
finite sets $P,Q\subset \F^2$ with $|P|=m$ and $|Q|=n$,
\[
 |\{(x,y)\in P\times Q:f(x;y)=0\}|
 \ll_t(mn)^{2/3}+m+n+\frac{mn}{p}.
\]
\end{corollary}

\begin{proof}
We work over $\kk$. First observe that a finite common zero set in
$\kk^2$ of polynomials of degree at most $t$ has at most $t^2$ points.
Indeed, choose one nonzero polynomial among them.
No irreducible factor of this polynomial divides all the others,
since otherwise the common zero set would contain a curve. A generic
linear combination of the remaining polynomials is therefore coprime
to the chosen polynomial, and B\'ezout's theorem applies. A nonzero
constant among the polynomials instead gives the empty set.

Let $\mathcal N_f$ be the set of parameters $y\in\kk^2$ for which
$f(\,\cdot\,;y)\equiv0$. This is the common zero set of the
coefficients of $f$ in the point variables, each of degree at most
$t$. If it contained a curve $V(h)$, then $h$ would divide every
coefficient, giving $f=hJ$ and making $f$ geometrically Cartesian.
A positive-dimensional closed subset of the affine plane contains
a curve. Hence, $\mathcal N_f$ is finite and $|\mathcal N_f|\le t^2$.

Next, fix an irreducible curve $\Lambda=V(g)$ that occurs in a nonzero
fiber. Let
\[
 T_\Lambda=\{y\in\kk^2:g\mid f(\,\cdot\,;y)\}.
\]
As in the division argument above, the remainder of $f$ modulo $g$
has coefficients in $\kk[y_1,y_2]$ of degree at most $t$: division
in the point variables only forms constant linear combinations of
their original coefficients. The common zero set of these remainder
coefficients is exactly $T_\Lambda$, including the zero fibers.
Thus, $T_\Lambda=F_\Lambda\cup\mathcal N_f$, where $F_\Lambda$ is defined
by the same occurrence condition as in Section~\ref{sec:structure}.
If $T_\Lambda$ contained a curve $V(h)$, the same argument would give
$f=gH+hJ$, contrary to the hypothesis. Thus, $T_\Lambda$ is finite,
and the preceding bound gives
\[
 |T_\Lambda|\le t^2.
\]

Let $\CC_Q$ be the set of all distinct irreducible components
of the nonzero fibers indexed by $Q\setminus \mathcal N_f$. Each such fiber
has at most $t$ components, so $n_{\CC}:=|\CC_Q|\le tn$. Each component
occurs for at most $t^2$ parameters. Counting incidences separately
on each component, we obtain
\begin{align*}
 |\{(x,y)\in P\times Q:f(x;y)=0\}|
 &\le t^2m+t^2 I(P,\CC_Q)\\
 &\ll_t m+(m n_{\CC})^{2/3}+m+n_{\CC}+\frac{m n_{\CC}}{p}\\
 &\ll_t(mn)^{2/3}+m+n+\frac{mn}{p},
\end{align*}
where the second line is Theorem~\ref{thm:pure}. If $\CC_Q$ is
empty, the first line already proves the result.
\end{proof}

The non-Cartesian hypothesis is necessary for a bound on
point--parameter pairs. For example, the polynomial
\[
 f(x_1,x_2;a,b)=x_1(x_2-a)+b
\]
is irreducible, because it is monic and linear in $b$, and is
$(x_1,b)$-Cartesian. If $A,B\subset \F$ are finite and
\[
 P=\{(0,v):v\in B\},\qquad Q=\{(u,0):u\in A\},
\]
then $f$ vanishes on all of $P\times Q$. The fibers indexed by $Q$
are distinct, but each is the union of the same vertical line
$x_1=0$ and a horizontal line $x_2=u$. Their distinct irreducible
components give at most $|B|+|A\cap B|$ incidences with $P$.
Thus, this example does not contradict Theorem~\ref{thm:pure}, which
counts each irreducible component once. Over an infinite field of
characteristic zero, taking $|A|=|B|$ arbitrarily large also shows why
the hypothesis cannot be dropped from Corollary~\ref{cor:noncartesian}.

Both the theorem in \cite{NMPVdZ} and
Corollary~\ref{cor:noncartesian} include the point--line relation
$x_2-a x_1-b=0$. This polynomial is non-Cartesian: otherwise a curve
in the parameter plane would index infinitely many distinct lines
all containing a fixed curve in the point plane, which is impossible.
Over $\mathbb C$, Corollary~\ref{cor:noncartesian} removes the
$\varepsilon$ loss in \cite[Theorem~1.3]{NMPVdZ}. It also gives a
bound over arbitrary fields, with the additional term $mn/p$.
Theorem~\ref{thm:pure} needs no non-Cartesian hypothesis because
it counts distinct components rather than parameters.

\subsection{Rich components}

We next bound the number of components containing many points of a
fixed set. The lower bound on $\kappa$ allows us to absorb the linear
and characteristic terms in Theorem~\ref{thm:pure}.

\Needspace{16\baselineskip}
\begin{corollary}[Rich components]\label{cor:richcomponents}
Let $\F$ have characteristic $p$, with $p=\infty$ in characteristic
zero. Suppose that $f\in \F[x_1,x_2,y_1,y_2]$ has total degree at most
$t$, and let $P\subset \F^2$ have $N\ge1$ elements. There is a constant
$A_t\ge1$ such that, for every integer
\[
 \kappa\ge A_t\max\{1,N/p\},
\]
the set
\[
 \CC_{\ge\kappa}(f;P)=\{\Lambda\in\family_{\F}(f):|P\cap \Lambda|\ge \kappa\}
\]
is finite and satisfies
\[
 |\CC_{\ge\kappa}(f;P)|\ll_t \frac{N^2}{\kappa^3}+\frac N\kappa.
\]
Here the curves are counted as distinct geometrically irreducible
components, even if a component occurs in several fibers. In particular,
when $N\le p$, the conclusion holds for every $\kappa\ge A_t$.
\end{corollary}

\begin{proof}
Let $C_t\ge1$ be a constant for Theorem~\ref{thm:pure}, and take
$A_t=4C_t$. For any finite subfamily
$\CC\subset\CC_{\ge\kappa}(f;P)$ of size $n$, that theorem gives
\[
 \kappa n\le I(P,\CC)
 \le C_t\left(N^{2/3}n^{2/3}+N+n+\frac{Nn}{p}\right).
\]
Since $C_t(1+N/p)\le \kappa/2$, the last two terms can be absorbed into
the left-hand side, giving
\[
 \kappa n\le2C_t\bigl(N^{2/3}n^{2/3}+N\bigr).
\]
If $n>0$, at least one of the two terms on the right is at least
$\kappa n/2$. Hence, either
\[
 n\le(4C_t)^3\frac{N^2}{\kappa^3}
 \qquad\text{or}\qquad
 n\le4C_t\frac N\kappa.
\]
The bound holds for every finite subfamily. Thus,
$\CC_{\ge\kappa}(f;P)$ is finite and satisfies the same bound;
otherwise it would contain finite subfamilies of arbitrarily large size.
\end{proof}

\subsection{Translates of a fixed curve}

The degree assumption in the following corollary ensures that
different translation vectors give different curves.

\begin{corollary}[Translates of a fixed curve]\label{cor:translates}
Let $\F$ be a field of characteristic $p$, with $p=\infty$ in
characteristic zero, and let $\Lambda_0\subset\overline{\F}^{\,2}$ be a
geometrically irreducible plane curve defined over $\F$, of degree
$D$ with $2\le D<p$. If $P,Q\subset \F^2$, $|P|=m$, and $|Q|=n$, then
\[
 |\{(q,y)\in P\times Q:q-y\in \Lambda_0\}|
 \ll_D (mn)^{2/3}+m+n+\frac{mn}{p}.
\]
\end{corollary}

\begin{proof}
Choose an absolutely irreducible defining polynomial
$F\in \F[X,Y]$ of degree $D$. We first show that $\Lambda_0+v=\Lambda_0$ is
impossible for $v\in \F^2\setminus\{0\}$. After an invertible linear
change of coordinates, such an equality would give
$V(F(X+1,Y))=V(F(X,Y))$ over $\overline{\F}$.
The two irreducible polynomials are then scalar multiples. Comparing
their highest homogeneous parts shows that
$F(X+1,Y)=F(X,Y)$. If $h=\deg_XF>0$, the coefficient of $X^{h-1}$
in $F(X+1,Y)-F(X,Y)$ is $h$ times the leading coefficient of $F$
as a polynomial in $X$. It is nonzero because $h\le D<p$.
Thus, $F$ depends only on $Y$. Absolute irreducibility then forces
$D=1$, a contradiction.

Consequently, $\{\Lambda_0+y:y\in Q\}$ consists of $n$ distinct curves.
These are members of the two-parameter family $f(x;y)=F(x-y)$,
whose total degree is $D$. Applying Theorem~\ref{thm:pure} proves
the assertion.
\end{proof}

\subsection{Polynomial values on difference sets}\label{sec:polynomialvalues}

We now extend the distinct-distance bound to polynomial functions.
Let $\F$ have characteristic $p$,
with the usual convention in characteristic zero, and write
$\kk=\overline{\F}$. For a nonconstant polynomial $F\in \F[U,V]$ of
total degree $D<p$ and a nonempty finite set $P\subset \F^2$, put
\[
 \Delta_F(P)=F(P-P)=\{F(x-y):x,y\in P\},\qquad N=|P|.
\]
The set $\Delta_F(P)$ includes all values, including zero, and $F$
need not be symmetric under $z\mapsto-z$. We say that $F$ is constant
on a line if its polynomial restriction to that line is constant.
All implied constants depend only on $D$.

\Needspace{10\baselineskip}
\begin{corollary}[General polynomial values]\label{cor:polynomialvalues}
Suppose $2\le D<p$ and
\begin{equation}\label{eq:poly-nonlinearform}
 F(U,V)\ne G(aU+bV)
 \quad\text{for all }G\in \F[T],\quad (a,b)\in \F^2\setminus\{0\}.
\end{equation}
If $P$ is not contained in a union of at most $D-1$ parallel affine
lines, then
\begin{equation}\label{eq:polyvalues}
 |\Delta_F(P)|\gg_D\min\{N^{2/3},p\}.
\end{equation}
In particular, $|\Delta_F(P)|\gg_D N^{2/3}$ when $N\le p^{3/2}$.
\end{corollary}

The fibers need not be irreducible or smooth. We estimate the
contribution of their nonlinear components by
Corollary~\ref{cor:translates} and treat their line components
separately. The following lemma gives the required bounds.

\begin{lemma}[Nonlinear components and constant lines]\label{lem:polyfibers}
Let $F\in \F[U,V]$ be nonconstant of degree $D<p$.
For $\rho\in \F$, let $\nu_{F,P}^{\mathrm{nl}}(\rho)$ count the ordered
pairs $(x,y)\in P^2$ for which $x-y$ lies on a nonlinear
geometrically irreducible component of $F-\rho=0$. Then
\begin{equation}\label{eq:poly-nonlinearpairs}
 \nu_{F,P}^{\mathrm{nl}}(\rho)
 \ll_D N^{4/3}+\frac{N^2}{p}.
\end{equation}
If~\eqref{eq:poly-nonlinearform} holds, there are at most $D$
geometric affine lines on which $F$ is constant, and at most $D-1$
such lines in any one direction.
\end{lemma}

\begin{proof}
For each nonlinear geometrically irreducible component $\Lambda$ of
$F-\rho=0$,
apply Corollary~\ref{cor:translates} over $\kk$, with both finite
sets equal to $P$. Its degree lies between $2$ and $D<p$, so
\[
 |\{(x,y)\in P^2:x-y\in\Lambda\}|
 \ll_D N^{4/3}+N+\frac{N^2}{p}.
\]
There are at most $D$ components. Summing over the nonlinear ones,
and absorbing $N$ into $N^{4/3}$, proves
\eqref{eq:poly-nonlinearpairs}. Working over $\kk$ allows components
that are not defined over $\F$. Pairs counted on more than one component
only increase the upper bound.

Assume now~\eqref{eq:poly-nonlinearform}. Every directional derivative
$\partial_vF=v_1F_U+v_2F_V$, with $v\in\kk^2\setminus\{0\}$,
is nonzero.
Indeed, a relation $v_1F_U+v_2F_V=0$ is a linear dependence between
coefficient vectors with entries in $\F$. If such a relation exists
over $\kk$, a nonzero one exists over $\F$. After an invertible
$\F$-linear change of coordinates, it says that the derivative in
one coordinate vanishes. The restriction $D<p$ then makes $F$
independent of that coordinate, contrary to
\eqref{eq:poly-nonlinearform}.

Call a line over $\kk$ a constant line of $F$ if the restriction
of $F$ to it is constant. The defining linear polynomial of each
constant line parallel to $v$ divides $\partial_vF$, which is
nonzero and has degree at most $D-1$. Hence, at most $D-1$ constant
lines have that direction. If all constant lines are parallel, this also proves the
total bound. Otherwise choose two nonparallel constant lines. Their
values agree at their intersection, say with value $\rho$. Every
other constant line meets at least one of them and so has the same
value. All constant lines are therefore distinct line components of
the single polynomial $F-\rho$, giving at most $D$ in total.
\end{proof}

\begin{proof}[Proof of Corollary~\ref{cor:polynomialvalues}]
Write
\begin{equation}\label{eq:poly-M}
 M_P=\max_{\ell}|P\cap\ell|,
\end{equation}
where $\ell$ ranges over affine $\F$-lines. A geometric line also
contains at most $M_P$ points of $P$: if it contains two $\F$-points,
it is defined over $\F$, and otherwise the assertion follows from
$M_P\ge1$. By Lemma~\ref{lem:polyfibers}, at most $DNM_P$ ordered
differences lie on any constant line of $F$.

If $M_P\le N/(2D)$, at least $N^2/2$ differences avoid all these
lines and hence lie on nonlinear components of their respective
fibers. Thus,
\[
 \frac{N^2}{2}
 \le\sum_{\rho\in\Delta_F(P)}\nu_{F,P}^{\mathrm{nl}}(\rho)
 \ll_D |\Delta_F(P)|\left(N^{4/3}+\frac{N^2}{p}\right).
\]
This gives~\eqref{eq:polyvalues}, since
\begin{equation}\label{eq:poly-harmonic}
 \frac{N^2}{N^{4/3}+N^2/p}
 =\frac1{N^{-2/3}+p^{-1}}
 \ge\frac12\min\{N^{2/3},p\}.
\end{equation}

Otherwise choose a line $\ell=a+\F v$ containing $M_P>N/(2D)$
points of $P$, with $a\in P$ and $v\in \F^2\setminus\{0\}$.
For $z\in P$, the polynomial $s\mapsto F(z-a-sv)$ is constant
exactly when $z-a+\kk v$ is a constant line of $F$.
There are at most $D-1$ constant lines parallel to $v$. Thus, the
points $z$ for which this restriction is constant lie on at most
$D-1$ parallel lines over $\kk$.
Each of those meeting $P$ is defined over $\F$, since it has a
$\F$-point and direction $v\in \F^2$. The hypothesis on $P$ supplies
a $z\in P$ for which $s\mapsto F(z-a-sv)$ is nonconstant.

The $M_P$ parameters with $a+sv\in P\cap\ell$ give at least
$M_P/D$ values: a nonconstant polynomial of degree at most $D$
takes any one value at at most $D$ parameters. All these values
belong to $\Delta_F(P)$, so
\[
 |\Delta_F(P)|\ge\frac{M_P}{D}>\frac{N}{2D^2}
 \ge\frac1{2D^2}\min\{N^{2/3},p\}.
\]
\end{proof}

The same argument gives a quantitative estimate for every nonconstant
polynomial, including polynomials in one linear form.

\begin{corollary}[A bound in terms of collinear points]\label{cor:polycollinear}
For every nonconstant $F\in \F[U,V]$ of degree $D<p$, with $M_P$
as in~\eqref{eq:poly-M},
\[
 |\Delta_F(P)|\gg_D
 \min\left\{N^{2/3},p,\frac{N}{M_P}\right\}.
\]
Consequently, $M_P\le N^{1/3}$ implies~\eqref{eq:polyvalues}
without condition~\eqref{eq:poly-nonlinearform}.
\end{corollary}
\begin{proof}
Each fiber $F-\rho$ has at most $D$ line components, each contributing
at most $NM_P$ ordered pairs. Together with
\eqref{eq:poly-nonlinearpairs}, this gives
\[
 N^2\ll_D |\Delta_F(P)|
 \left(N^{4/3}+\frac{N^2}{p}+NM_P\right).
\]
Divide by $N^2$ and bound the resulting sum by three times its
largest term. This also covers characteristic zero, where the
term $p^{-1}$ is omitted.
\end{proof}

Using both orientations of a difference improves the condition on
$P$ for several natural classes of polynomials.

\begin{proposition}[A criterion using parallel lines]\label{prop:polytriple}
Let $2\le D<p$. Suppose there are no independent $v,w\in \F^2$
such that $F$ is constant on each of the three lines
\begin{equation}\label{eq:polytriple}
 \F v,\qquad w+\F v,\qquad -w+\F v.
\end{equation}
If $P$ is not contained in an affine line with direction $v\ne0$
such that $s\mapsto F(sv)$ is constant, then~\eqref{eq:polyvalues}
holds.
\end{proposition}
\begin{proof}
The hypothesis implies~\eqref{eq:poly-nonlinearform}: otherwise
take $v$ in the kernel of the linear form and $w$ independent of
$v$. Thus, Lemma~\ref{lem:polyfibers} applies. If $M_P\le N/(2D)$,
the first case of the proof of Corollary~\ref{cor:polynomialvalues}
gives the conclusion. Otherwise, take a line $\ell=a+\F v$
containing $M_P$ points of $P$, with $a\in P$. If $s\mapsto F(sv)$
is nonconstant, the differences
$x-a$, $x\in P\cap\ell$, give at least $M_P/D$ values.

If that restriction is constant, choose $z\in P\setminus\ell$,
which exists by hypothesis, and put $w=z-a$. The vectors $v,w$
are independent. At least one of
\[
 s\longmapsto F(w-sv),\qquad s\longmapsto F(-w+sv)
\]
is nonconstant, by~\eqref{eq:polytriple}. At the parameters with
$a+sv\in P\cap\ell$, these are values at the ordered differences
$z-x$ and $x-z$. Root counting again yields
$|\Delta_F(P)|\ge M_P/D>N/(2D^2)$.
\end{proof}

\begin{corollary}[Homogeneous polynomial values]\label{cor:homogeneousvalues}
Let $H\in \F[U,V]$ be homogeneous of degree $2\le D<p$ and not a
scalar multiple of a power of one linear form over $\kk$.
If $P$ is not contained in an affine line with direction $v\ne0$
satisfying $H(v)=0$, then
\[
 |H(P-P)|\gg_D\min\{N^{2/3},p\}.
\]
For $N\ge4D^2$, the same estimate holds for
$|H(P-P)\setminus\{0\}|$.
\end{corollary}
\begin{proof}
Factor $H=c\prod_{j=1}^h L_j^{e_j}$ over $\kk$, where $h\ge2$
and the linear forms $L_j$ are pairwise nonproportional. If
$H(a+sv)$ is a nonzero constant, every factor $L_j(a+sv)$ must
be constant. Thus, $L_j(v)=0$ for all $j$, which forces $v=0$.
If $H(a+sv)$ is identically zero, one factor is identically zero,
so the line equals $\ker L_j$ for some $j$. Hence, the constant
lines of $H$ are precisely its zero lines through the origin.
There is at most one in any given direction, and
$H(sv)=s^D H(v)$ is constant exactly when $H(v)=0$.
Proposition~\ref{prop:polytriple} proves the first assertion.

For the second, suppose $N\ge4D^2$. If $M_P\le N/(2D)$, the
at most $D$ zero lines account for at most $DNM_P\le N^2/2$
ordered pairs. Every nonzero fiber has only nonlinear components,
so~\eqref{eq:poly-nonlinearpairs} and~\eqref{eq:poly-harmonic}
give the result using only nonzero values.

If $M_P>N/(2D)$, choose $\ell=a+\F v$ with $a\in P$ and
$|P\cap\ell|=M_P$. When $H(v)\ne0$, use the nonconstant
polynomial $s\mapsto H(sv)$. When $H(v)=0$, take
$z\in P\setminus\ell$; the line $z-a+\kk v$ is not a zero line
through the origin and hence is not a constant line of $H$.
Thus, $s\mapsto H(z-a-sv)$ is nonconstant. In either case we
evaluate a polynomial of degree at most $D$ at $M_P$ parameters.
At most $D$ parameters give zero, and every other value has at
most $D$ preimages. Since $M_P>N/(2D)\ge2D$,
\[
 |H(P-P)\setminus\{0\}|\ge\frac{M_P-D}{D}
 \ge\frac{M_P}{2D}>\frac{N}{4D^2}.
\]
\end{proof}

In particular, this corollary applies to $U^d+V^d$ for every
$2\le d<p$, and to $U^aV^b$ for $a,b\ge1$ and $a+b<p$.
If $H(v)\ne0$ for every $v\in \F^2\setminus\{0\}$, the assertion
about all values has no restriction on $P$.

\begin{corollary}[Polynomials of degree two or three]\label{cor:lowdegreevalues}
Suppose $2\le D\le3$, $D<p$, and~\eqref{eq:poly-nonlinearform}
holds. Then~\eqref{eq:polyvalues} holds whenever $P$ is not
contained in a line with direction $v\ne0$ such that $F(sv)$ is
constant. In particular, it holds for every noncollinear $P$,
with an absolute implied constant.
\end{corollary}
\begin{proof}
If the three lines in~\eqref{eq:polytriple} were constant lines,
the nonzero polynomial $\partial_vF$ would have three distinct
line factors. This is impossible because its degree is at most
$D-1\le2$. The lines are distinct since the characteristic is
different from $2$. Apply Proposition~\ref{prop:polytriple}.
\end{proof}

\begin{remark}[Fibers without line components]\label{rem:polynolines}
If $F$ has no constant geometric affine line, then
\eqref{eq:polyvalues} holds for every nonempty $P$.
Indeed, all fiber components are nonlinear, so summing
\eqref{eq:poly-nonlinearpairs} gives
$N^2\ll_D|\Delta_F(P)|(N^{4/3}+N^2/p)$.
For example, $F(U,V)=U^2+V^3$ has this property when $p>3$.
On a line $a+sv$, its cubic coefficient is $v_2^3\ne0$ if
$v_2\ne0$, and otherwise its quadratic coefficient is
$v_1^2\ne0$.
\end{remark}

\begin{remark}[Necessary geometric hypotheses]\label{rem:polyobstructions}
The restrictions above cannot be omitted. For $F(U,V)=U$ and
$P=A\times A$, where $A=\{0,\ldots,h-1\}\subset\Fp$ and
$2h<p$, one has $N=h^2$ but $|F(P-P)|=2h-1<2N^{1/2}$.
Even~\eqref{eq:poly-nonlinearform} and noncollinearity do not
suffice in degree four. For $p>4$, take
\[
 F(U,V)=UV(V-1)(V+1),\qquad P=\Fp\times\{0,1\}.
\]
Then $F(P-P)=\{0\}$, since the second coordinate of every
difference belongs to $\{0,1,-1\}$. The derivatives
$F_U=V^3-V$ and $F_V=U(3V^2-1)$ are linearly independent,
so~\eqref{eq:poly-nonlinearform} holds. Moreover, each nonzero
fiber is geometrically irreducible: $U(V^3-V)-\rho$ is primitive
and linear in $U$ over $\kk(V)$. It is smooth because its
$U$-derivative is nonzero on that fiber. Thus, smoothness and
irreducibility of the nonzero fibers do not remove this obstruction.
More generally, if the three constant lines in~\eqref{eq:polytriple}
exist over $\Fp$, the set
$P=\Fp v\cup(w+\Fp v)$ has $2p$ points and $|F(P-P)|\le3$.
\end{remark}

For the quadratic form $H(U,V)=U^2+V^2$, Murphy, Petridis, Pham,
Rudnev and Stevens~\cite[Theorem~3]{MPPRS} proved the pinned estimate
\[
 \max_{y\in P}|\{H(x-y):x\in P\}|\gg N^{2/3}
\]
over an arbitrary field of characteristic $p$, provided that
$N\le p^{4/3}$ and $P$ is not contained in an isotropic line,
that is, a line whose nonzero direction vector $v$ satisfies $H(v)=0$.
A pinned estimate requires many values from a single point $y$;
our estimates concern the full set $F(P-P)$.

Corollaries~\ref{cor:polynomialvalues} and~\ref{cor:homogeneousvalues}
apply to general polynomials under the stated geometric hypotheses.
They give the exponent $2/3$ for $N\le p^{3/2}$. For the quadratic
form, Corollary~\ref{cor:homogeneousvalues} extends the characteristic
range in \cite[Theorem~3]{MPPRS} for the unpinned problem over
arbitrary fields.

Over $\Fp$, Murphy et al.~\cite[Theorem~1]{MPPRS} also proved
that some pin determines $\gg p$ quadratic distances when
$N\ge p^{5/4}$. Together with \cite[Theorem~3]{MPPRS}, this already
covers the prime-field range considered here for the usual distance.
Thus, our contribution to this comparison is the extension to
polynomial functions and the wider characteristic range over
arbitrary fields, rather than a new prime-field range for quadratic
distances.

\begin{remark}
Thang Pham pointed out to the author that combining the methods
of~\cite{Iosevichetal,Io2,MPPRS} with Lewko's point--line incidence
bound gives the following estimate. If $P\subset\Fp^2$,
$p\equiv3\pmod4$, and $|P|\le p^{4/3}$, then the number of
isosceles triangles in $P$ with nonzero bases is
\[
 \ll \frac{|P|^3}{p}+|P|^{11/5}.
\]
Consequently, the number of distinct distances determined by $P$
is
\[
 \gg \min\{p,|P|^{4/5}\}.
\]
In particular, for $|P|\le p^{5/4}$, this gives the lower bound
$\gg|P|^{4/5}$ for the usual quadratic distance, improving the
exponent $2/3$ in Corollary~\ref{cor:polynomialvalues} in this case.
The threshold $|P|\ge p^{5/4}$ for determining a positive proportion
of all distances was already established by Murphy
et al.~\cite[Theorem~1]{MPPRS}.
\end{remark}

\subsection{Polynomial expansion}

Bukh and Tsimerman \cite[Theorem~1]{BT} proved that, for
$F\in\Fp[X]$ of degree $2\le D<p$ and $A\subset\Fp$ with
$N=|A|\le p^{1/2}$,
\[
 |A+A|+|F(A)+F(A)|\gg_D N^{1+1/(16\cdot6^D)}.
\]
The next corollary gives the exponent $5/4$, independent of $D$,
for $N\le p^{2/3}$. It follows by applying
Theorem~\ref{thm:pure} to a Cartesian product of two sumsets.

\begin{corollary}[Polynomial expansion]\label{cor:expansion}
Let $F\in\Fp[X]$ have degree $D$, where $2\le D<p$, and let
$A\subset\Fp$ have cardinality $N\ge1$. Then
\begin{equation}\label{eq:polynomialexpansion}
 |A+A|\,|F(A)+F(A)|
 \gg_D \min\{N^{5/2},pN\}.
\end{equation}
In particular, if $N\le p^{2/3}$, then
\[
 \max\{|A+A|,|F(A)+F(A)|\}\gg_D N^{5/4}.
\]
\end{corollary}

\begin{proof}
Consider the point set
\[
 P=(A+A)\times(F(A)+F(A))
\]
and the curves
\[
 \Lambda_{a,b}:\quad x_2=F(x_1-a)+b,
 \qquad (a,b)\in A\times F(A).
\]
These are geometrically irreducible graphs in the two-parameter family
defined by $x_2-F(x_1-a)-b$, which has total degree $D$.
They are also distinct. Indeed, write
$F(X)=c_DX^D+c_{D-1}X^{D-1}+\cdots$, with $c_D\ne0$.
The coefficient of $x_1^{D-1}$ in $F(x_1-a)+b$ is
$c_{D-1}-Dc_Da$. Since $D<p$, equality of two graphs forces their
parameters $a$ to coincide, and then their parameters $b$ coincide.

Put
\[
 m=|P|=|A+A|\,|F(A)+F(A)|,
 \qquad n=N|F(A)|.
\]
A polynomial of degree $D$ takes any prescribed value at at most $D$
points, so
\[
 N^2/D\le n\le N^2,\qquad n\ge N.
\]
For each $(a,b)\in A\times F(A)$, the $N$ distinct points
\[
 (a+u,F(u)+b),\qquad u\in A,
\]
belong to both $P$ and $\Lambda_{a,b}$. Hence, Theorem~\ref{thm:pure}
gives, for a constant $C_D\ge1$,
\[
 Nn\le C_D\bigl((mn)^{2/3}+m+n+mn/p\bigr).
\]
If $N\ge2C_D$, the term $C_Dn$ can be absorbed into the left-hand
side. At least one of the other three terms inside the parentheses is
then at least $Nn/(6C_D)$. The corresponding alternatives are
\[
 m\ge (6C_D)^{-3/2}N^{3/2}\sqrt n,
 \qquad m\ge (6C_D)^{-1}Nn,
 \qquad m\ge (6C_D)^{-1}pN.
\]
Since $n\ge N$, we have $Nn\ge N^{3/2}\sqrt n$. Using also
$n\ge N^2/D$, we obtain
\[
 m\gg_D\min\{N^{3/2}\sqrt n,pN\}
 \gg_D\min\{N^{5/2},pN\}.
\]
If $N<2C_D$, the same conclusion follows, after decreasing the
constant, from $m\ge N$. This proves
\eqref{eq:polynomialexpansion}. Finally, $N\le p^{2/3}$ implies
$N^{5/2}\le pN$, and the larger of two nonnegative quantities is at
least the square root of their product.
\end{proof}

\begin{remark}
When $D=2$, the bijection $(x_1,x_2)\mapsto(x_1,x_2-F(x_1))$ sends the curves
$\Lambda_{a,b}$ to lines. Thus, the quadratic case of
Corollary~\ref{cor:expansion} already follows from Lewko's point--line
theorem~\cite{Lewko}.
\end{remark}

\subsection{Sharpness and the number of parameters}

\paragraph{Sharpness for a nonlinear family.}
The map $(x_1,x_2)\mapsto(x_1,x_2+x_1^2)$ is a polynomial automorphism and sends
the lines $x_2=ax_1+b$ to the parabolas $x_2=x_1^2+ax_1+b$.
It preserves incidences, so the point--line examples showing sharpness
of the Szemer\'edi--Trotter theorem~\cite{ST} transfer to this nonlinear family.
For example, take the points $[R]\times[2R^2]$ and the lines with
$a\in[R]$, $b\in[R^2]$, where $[s]=\{1,\ldots,s\}$.
There are $2R^3$ points, $R^3$ lines, and $R^4$ incidences.
This construction works over the reals and, by reduction, over
$\Fp$ with $p>4R^3$. Applying the automorphism proves that the
exponent $2/3$ in Theorem~\ref{thm:pure} cannot be decreased even
for these parabolas. In odd characteristic they are exactly the
translates of $x_2=x_1^2$, since
$(x_1-u)^2+v=x_1^2-2ux_1+u^2+v$.

The same family shows that the other terms are also necessary.
All $p^2$ points
of $\Fp^2$ and all $p^2$ parabolas $x_2=x_1^2+ax_1+b$ have $p^3$
incidences, which requires the characteristic term. One parabola
containing many points requires the term $m$, while many distinct
parabolas through one point require the term $n$.
These examples establish sharpness for the general incidence theorem.

\begin{remark}\label{rem:twoparameters}
A two-parameter family need not have two degrees of freedom. The latter
condition requires, in particular, a uniform bound on the number of
curves through any two distinct points. For example, over an infinite
field the family
\[
 x_2=(a+b x_1)x_1(x_1-1),\qquad (a,b)\in \F^2,
\]
consists of distinct geometrically irreducible curves, all containing
$(0,0)$ and $(1,0)$. Theorem~\ref{thm:pure} includes such families.
\end{remark}

\paragraph{The number of parameters.}
Allowing three parameters changes the problem. Indeed, the same
bound fails in that setting, even for a $K_{3,3}$-free graph defined
by one polynomial of total degree $3$. To see this, let $R\ge4$
be an integer and choose a prime $p>R^4$. In $\Fp$, using the
integer representatives, take
\[
 P=[R]\times[3R^3],\qquad
 Q=[R]\times[R^2]\times[R^3].
\]
Define the bipartite graph $G=(P,Q,E)$ by joining $(x_1,x_2)$ to
$(a,b,c)$ when $x_2=ax_1^2+bx_1+c$.
Every parameter gives a distinct irreducible parabola and has exactly
$R$ neighbors in $P$: for $1\le x_1\le R$ the integer value of
$ax_1^2+bx_1+c$ lies between $1$ and $3R^3<p$.
Thus,
\[
 m=3R^4,\qquad n=R^6,\qquad |E|=R^7.
\]
Two distinct parabolas have at most two common points, so this graph
is $K_{3,3}$-free. On the other hand,
\[
 (mn)^{2/3}+m+n+\frac{mn}{p}=O(R^{20/3}).
\]
Letting $R$ tend to infinity rules out an estimate with these exponents
and a uniform constant for three-parameter families. In particular,
families such as all circles, whose centers and radii vary
independently, require a separate incidence theorem.

\appendix
\section{The degree restriction in the contact bound}\label{sec:cutoff}

The strict inequality $d<p/t$ in Lemma~\ref{lem:norm} cannot in general
be replaced by $d\le\lceil p/t\rceil$.

\begin{proposition}\label{prop:cutoff}
For every integer $t\ge2$ and every prime $p>1+4t^2$, there are
polynomials $\Phi$ and $F$ over $\overline{\Fp}$ with
$\deg\Phi=t$ and $\deg F=\lceil p/t\rceil$ such that all the hypotheses
of Lemma~\ref{lem:norm}, except $d<p/t$, hold for arbitrarily many
distinct marked points, and the contact order at each of them is $p$.
Consequently, the conclusion of that lemma fails for these configurations.
\end{proposition}

\begin{proof}
Work over $\Omega=\overline{\Fp}$ and set
\[
 \begin{gathered}
 d=\lceil p/t\rceil,\qquad \delta=td-p\in\{1,\ldots,t-1\},\\
 F=X^d-Y^{d-1},\qquad H=X^\delta Y^{t-\delta},\qquad \Phi=H-B.
 \end{gathered}
\]
The polynomial $F$ is irreducible since $d$ and $d-1$ are coprime.
The curve $V(F)\cap(\Omega^\times)^2$ has the parametrization
\[
 (X,Y)=(s^{d-1},s^d),\qquad s=Y/X,\qquad s\in\Omega^\times.
\]
The restriction of $H$ there is $s^{td-\delta}=s^p$.
For each $s_0\ne0$, let $q_{s_0}=(s_0^{d-1},s_0^d)$ and consider the
fiber $H=s_0^p$. This fiber is smooth at $q_{s_0}$ because both
exponents of $H$ lie strictly between $0$ and $p$. It is irreducible:
$\gcd(\delta,t-\delta)=\gcd(p,t)=1$, so its defining binomial is
irreducible in $\Omega[X^{\pm1},Y^{\pm1}]$, and it has no
coordinate-axis component. It is not a
component of $F$.

At $q_{s_0}$ the intersection multiplicity is
\[
 I_{q_{s_0}}(F,H-s_0^p)
 =\ord_{s=s_0}(s^p-s_0^p)=p.
\]
Since this fiber is smooth, this is precisely its contact order
with $F$. The discriminant in $B$ and the leading coefficient of $\Phi$
are nonzero constants. Moreover, $\Phi(X,Y;c)\not\equiv0$ for every
$c\in\Omega$, and
$\Phi(q_{s_0};B)=s_0^p-B$ is not identically zero. We may choose
arbitrarily many distinct $s_0$ in $\Omega$.
For primes $p>1+4t^2$, each summand $\epspart{p-(1+4t^2)}$ is positive,
so no bound depending only on $t,d$ can hold for their sum. The only
failed hypothesis of Lemma~\ref{lem:norm} is $d<p/t$.
Here $H$ restricts to the nonconstant $p$th power $s^p$ on $V(F)$,
so its derivative vanishes. The function-field degree estimate in
the proof of Lemma~\ref{lem:norm} excludes this possibility when
$d<p/t$.
\end{proof}

This example concerns the degree restriction in Lemma~\ref{lem:norm},
not the incidence bound in Theorem~\ref{thm:pure}. Indeed,
$H(q_s)=s^p$, so $q_s$ lies on the chosen member $H=s_0^p$ if and only
if $s=s_0$. For any finite set of the selected points and the
corresponding members, there is exactly one incidence per point.
In the proof of Theorem~\ref{thm:pure}, the interpolation degree is
kept below $p/t$; this restriction is accounted for by the term $mn/p$.

\section*{Acknowledgements} The author is grateful to Thang Pham for valuable comments on the distance
problem and for pointing out relevant references.
\section*{LLM Use}

The author used LLM models to assist with literature review,
drafting and revising the exposition, and developing and checking proofs.
The author takes full responsibility for the content of this paper,
including the correctness of its mathematical proofs and the accuracy
of its references.

\end{document}